\RequirePackage{fix-cm}
\documentclass[smallextended]{svjour3}       % onecolumn (second format)
\smartqed  % flush right qed marks, e.g. at end of proof
\usepackage{graphicx}

\usepackage{bm}
\usepackage{tikz}
\usepackage{pgf}
\usepackage{indentfirst}
\usepackage{graphicx}
\usepackage{graphics}
\usepackage[toc,page,title,titletoc,header]{appendix}
\usepackage{bm}
\usepackage{listings}  
\usepackage{amsmath}
\usepackage{setspace} 
\usepackage{indentfirst}
\usepackage{caption}
\usepackage{multirow} 
\usepackage{lipsum,multicol}
\usepackage{pdfpages}
\usepackage{float}
\usepackage{pdfpages}
\usepackage{url}
\usepackage{colortbl}
\usepackage{subfigure}
\usepackage{epsfig}
\usepackage{epstopdf}
\usetikzlibrary{shapes.geometric, arrows}
\usepackage{fancyhdr}
\usepackage{abstract}
\usepackage{tikz}
\usetikzlibrary{shapes.geometric, arrows}
\usepackage{amssymb}
\usepackage{latexsym}
\usepackage{verbatim}
\usepackage[numbers]{natbib} 
\usepackage{booktabs}
\usepackage{xspace}
\usepackage{paralist}
\usepackage{tabularx}
\usepackage{csquotes}
\usepackage{amssymb}% http://ctan.org/pkg/amssymb
\usepackage{pifont}% http://ctan.org/pkg/pifont
\usepackage[figuresright]{rotating}

\newcommand{\cmark}{\ding{51}}%
\newcommand{\xmark}{\ding{55}}%

\newcommand{\inner}[1]{\left\langle #1 \right\rangle}
\newcommand{\norm}[1]{\left\Vert #1\right\Vert}
\newcommand{\bb}[1]{\mathbb{#1}}
\newcommand{\M}{\mathcal{M}}

\newcommand{\ca}[1]{\mathcal{#1}}
\newcommand{\tr}[0]{\mathrm{tr}}

\newcommand{\tp}{^\top}

\newcommand{\A}{\ca{A}}

\newcommand{\Tx}{{\ca{T}_x}}
\newcommand{\Nx}{\ca{N}_x}

\newcommand{\xk}{{x_{k} }}

\newcommand{\Jc}{{{J}_c}}
\newcommand{\Ja}{{{J}_{\A}}}

\newcommand{\DJa}{ {\ca{D}_{{J}_\A}} }
\newcommand{\DJc}{{\ca{D}_{{J}_c}}}

\newcommand{\grad}{{\mathit{grad}\,}}
\newcommand{\hess}{{\mathit{hess}\,}}

\newcommand{\Pkg}[1]{\textsf{#1}\xspace}
\newcommand{\codeobj}[1]{\texttt{#1}\xspace}

\newcommand{\CDOpt}{\Pkg{CDOpt}}

\newcommand{\betaesti}{\tilde{\beta}}

\newcommand{\Lf}{ {M_{x,f}} }

\newcommand{\Lg}{ {L_{x,g}} }

\newcommand{\Lsc}{ \sigma_{x,c} }

\newcommand{\Mc}{ {M_{x,c}} }

\newcommand{\Ma}{ M_{x,A} }

\newcommand{\Lc}{ {L_{x,c}} }

\newcommand{\La}{ {L_{x,A}} }

\newcommand{\Lac}{ {L_{x, b}} }

\newcommand{\Omegax}[1]{ {\Omega_{#1}} }

\newcommand{\BOmegax}[1]{ {\bar{\Omega}_{#1}} }

\newcommand{\y}{{y}}
\newcommand{\X}{{\mathcal{X}}}

\newtheorem{theo}{Theorem}
\newtheorem{lem}{Lemma}
\newtheorem{prop}{Proposition}

\newtheorem{defin}{Definition}
\newtheorem{rmk}{Remark}
\newtheorem{assumpt}{Assumption}
\usepackage[figuresright]{rotating}
\usepackage{pdflscape}

\DeclareMathOperator*{\argmin}{arg\,min}

\usepackage{caption}
\usepackage{algorithm}
\usepackage{algpseudocode}
\usepackage{longtable}
\usepackage{appendix}

\usepackage{listings}
\usepackage{color}
\definecolor{deepblue}{rgb}{0,0,0.5}
\definecolor{deepred}{rgb}{0.6,0,0}
\definecolor{deepgreen}{rgb}{0,0.5,0}
\DeclareFixedFont{\ttm}{T1}{txtt}{m}{n}{7}
\newcommand\pythonstyle{\lstset{
		language=Python,
		basicstyle=\linespread{0.5}\selectfont\ttm,
		otherkeywords={self},
		keywordstyle=\ttm\color{deepblue},
		emph={MyClass,__init__},
		emphstyle=\ttm\color{deepblue},
		stringstyle=\color{deepred},
		commentstyle=\ttm\color{deepgreen},
		frame=tb,
		showstringspaces=false
}}
\lstdefinestyle{mystyle}{
	language=Python,
	basicstyle=\linespread{0.5}\selectfont\ttm,
	otherkeywords={self},
	keywordstyle=\ttm\color{deepblue},
	emph={__init__},
	emphstyle=\ttm\color{deepblue},
	stringstyle=\color{deepred},
	commentstyle=\ttm\color{deepgreen},
	frame=tb,
	showstringspaces=false
}
\usepackage[english]{babel}
\usepackage[utf8]{inputenc}
\usepackage{tikz} 
\usetikzlibrary{arrows,decorations.pathmorphing,backgrounds,fit,positioning,shapes.symbols,chains}
\newcommand\pythonexternal[2][]{{
		\pythonstyle
		\lstinputlisting[#1]{#2}}}

\begin{document}
	
	\title{CDOpt: A Python Package for a Class of Riemannian Optimization}
	
	
	
	\author{Nachuan Xiao    \and
		Xiaoyin Hu
		\and Xin Liu
		\and Kim-Chuan Toh %etc.
	}
	
	
	\institute{Nachuan Xiao \at
		The Institute of Operations Research and Analytics, National University of Singapore, Singapore.\\
		\email{xnc@lsec.cc.ac.cn}           %  \\
		\and
		Xiaoyin Hu \at
		School of Computer and Computing Science, Hangzhou City University, Hangzhou, China.\\
		Academy of Edge Intelligence, Hangzhou City University, Hangzhou(Gongshu), China.\\
		\email{hxy@amss.ac.cn} 
		\and
		Xin Liu \at
		State Key Laboratory of Scientific and Engineering Computing, Academy of Mathematics and Systems Science, Chinese Academy of Sciences, and University of Chinese Academy of Sciences, China.\\
		\email{liuxin@lsec.cc.ac.cn}
		\and
		Kim-Chuan Toh \at
		Department of Mathematics, and Institute of Operations Research and Analytics, National University of Singapore, Singapore.\\
		\email{mattohkc@nus.edu.sg}.
	}
	
	\date{Received: date / Accepted: date}

	\maketitle
	
	\begin{abstract}
		Optimization over an embedded submanifold defined by equality constraints $c(x) = 0$ has attracted much interest over the past few decades due to its wide applications in various areas, including computer vision, signal processing, numerical linear algebra, and deep learning. For solving such problems, many related optimization packages have been developed based on Riemannian optimization approaches, which rely on some basic geometrical materials of Riemannian manifolds, including Riemannian gradients, retractions, vector transports, etc. These geometrical materials can be challenging to determine in general. In fact, existing packages only accommodate a few well-known manifolds whose geometrical materials are more easily accessible. For other manifolds that are not contained in these packages, the users have to develop the geometric materials by themselves. In addition, it is not always tractable to adapt the advanced features from various state-of-the-art unconstrained optimization solvers to Riemannian optimization approaches.

		Here we introduce a user-friendly Python package, \CDOpt (available at \url{https://cdopt.github.io/} under BSD 3-clause license),  for solving a class of Riemannian optimization problems. \CDOpt is designed to complement existing Riemannian optimization packages by transforming Riemannian optimization problems into their unconstrained counterparts through
		the constraint dissolving approach. We prove that the when the penalty parameter in the constraint dissolving approach is 
		sufficiently large, Riemannian optimization problems and their unconstrained counterparts are equivalent. Therefore, solving Riemannian optimization problems through \CDOpt can directly benefit from various existing solvers and the rich expertise gained over the past few decades for unconstrained optimization. 
		Moreover, all the computations in \CDOpt related to any manifold in question are conducted through its constraints expression, hence users can easily define new manifolds in \CDOpt without any background on differential geometry. Furthermore, \CDOpt extends the neural layers from \Pkg{PyTorch} and \Pkg{Flax}, thus allowing users to train manifold constrained neural networks directly by the solvers for unconstrained optimization. Extensive numerical experiments demonstrate that \CDOpt is highly efficient and robust in solving various classes of Riemannian optimization problems.
	\end{abstract}

	\section{Introduction}
	\subsection{Problem formulation}
	In this paper, we consider the following constrained optimization problem 
	\begin{equation}
		\label{Prob_Ori}
		\tag{OCP}
		\begin{aligned}
			\min_{x \in \bb{R}^n} \quad & f(x)\\
			\text{s.t.} \quad & \M := \{x \in \bb{R}^n: c(x) = 0 \},
		\end{aligned}
	\end{equation}
	where the objective function  $f: \bb{R}^n \to \bb{R}$ and constraint mapping $c: \bb{R}^n \to \bb{R}^p$ of \ref{Prob_Ori} satisfy the following assumptions:
	\begin{assumpt}{\bf Blanket assumptions}
		\label{Assumption_1}
		
		\begin{enumerate}
			\item $f$ is locally Lipschitz continuous in $\bb{R}^n$;
			\item The transposed Jacobian of $c$, denoted as $\Jc(x) \in \bb{R}^{n \times p}$, 
			is locally Lipschitz continuous in $\bb{R}^{n}$;
			\item The relaxed constant rank constraint qualification (RCRCQ) \cite{minchenko2011relaxed} holds for any $x \in \M$, i.e. there exists a constant $p'\leq p$ and a neighborhood $\X$ of $\M$ such that  the rank of $\Jc(x)$ equals to $p'$ for any $x \in \X$. 
		\end{enumerate}
	\end{assumpt}
	Since $c$ is smooth in $\bb{R}^n$, the feasible region $\M$ is a closed submanifold embedded in $\bb{R}^n$ \cite{Absil2009optimization,boumal2020introduction}.
	In fact, \ref{Prob_Ori} satisfying Assumption \ref{Assumption_1} covers a great number of practically interesting Riemannian optimization problems  \cite{hu2020brief,xiao2022constraint} arising over the past few decades. These Riemannian optimization problems can be categorized into two different categories. One is the so-called {\it standard Riemannian optimization problem}, where the function value and derivatives of the objective function are explicitly formulated and affordable to evaluate, hence can be directly provided to optimization packages. These problems include the Kohn-Sham total energy minimization \cite{bai2012minimization}, dictionary learning \cite{zhai2020complete}, dual principal component pursuit \cite{tsakiris2018dual,hu2020anefficiency}, symplectic eigenvalue problems \cite{son2021symplectic}, etc. Interested readers may refer to the books \cite{Absil2009optimization, boumal2020introduction} and a recent survey paper \cite{hu2020brief} for more details of these problems. 
	
	Another category is referred as {\it training neural networks with manifold constraints} throughout this paper. Training deep neural networks (DNNs) is usually thought to be challenging both theoretically and practically, where the gradient vanishing and exploding problem is one of the most important reasons \cite{glorot2010understanding}.  To address this issue, several recent works focus on training DNNs while imposing manifold constraints to the weights of their neural layers, especially on restricting the weights over the Stiefel manifold \cite{arjovsky2016unitary,bansal2018can,huang2018orthogonal,lezcano2019trivializations}. As orthogonality can be used to impose energy preservation properties on these DNNs  \cite{zhou2006special}, some existing works demonstrate that the orthogonal constraints can stabilize the distribution of activations among neural layers within DNNs and make their optimization more efficient. Moreover, many existing works \cite{bansal2018can,huang2018orthogonal,wang2020orthogonal} observe encouraging improvements in the accuracy and robustness of the DNNs with orthogonally constrained weights.  
	
	In training neural networks with manifold constraints, the neural network is built by the composition of neural layers, while the objective function is formulated as the expectation of the losses over samples. Therefore, the function value and differentials of the objective function are usually unaffordable to be evaluated exactly, and most of the {\it optimizers} (i.e., solvers for training neural networks) are developed based on stochastic optimization algorithms to reduce the computational costs. Moreover,  existing deep-learning frameworks (e.g. \Pkg{PyTorch} \cite{pytorch2019pytorch}, \Pkg{JAX} \cite{jax2018github}, \Pkg{TensorFlow} \cite{tensorflow2015-whitepaper}) provide various advanced features, including GPU/TPU acceleration, automatic differentiation (AD), just-in-time compilation, to further accelerate the training. As a result, training neural networks with manifold constraints requires specialized optimization solvers (i.e., optimizers) to utilize the advanced features of these frameworks, which leads to great differences between their optimizers and the solvers for standard Riemannian optimization problems.

	\subsection{Existing approaches and optimization packages}
	Due to the diffeomorphisms between an Euclidean space and a Riemannian manifold, various {\it unconstrained optimization approaches} (i.e., approaches for solving unconstrained nonconvex optimization) can be transferred to their corresponding {\it Riemannian optimization approaches} (i.e., the approaches for Riemannian optimization). In particular, \cite{Absil2009optimization} provides several well-recognized frameworks based on geometrical materials from differential geometry, including geodesics,  parallel transports, and Riemannian differentials. 
	
	The geodesics generalize the concept of straight lines to Riemannian manifolds, but are expensive to compute in most cases. To this end, \cite{Absil2009optimization} provides the concept of retractions as relaxations to geodesics, which makes them more affordable to compute. Besides,  parallel transports are mappings that move tangent vectors along given curves over a Riemannian manifold ``parallelly''. Computing parallel transports is essential in computing the linear combinations of two vectors from different tangent spaces. In particular, it is necessary for those approaches that utilize information in past iterations (e.g., quasi-Newton methods, nonlinear conjugate gradient methods, momentum SGD, ADAM).  For various Riemannian manifolds, computing the parallel transports amounts to solving differential equations,  which is generally unaffordable in practice \cite{Absil2009optimization}. Therefore, the concept of vector transports is proposed as approximations to parallel transports to alleviate the computational cost in Riemannian optimization approaches. Although retractions and vector transports can make the computation more affordable by approximation, how to develop efficient formulations of retractions and vector transports for new manifold constraints remain challenging.

	Almost all the existing Riemannian optimization packages are developed based on the framework proposed by \cite{Absil2009optimization}. These packages are either developed for standard Riemannian optimization problems (e.g., \Pkg{Manopt} \cite{boumal2014manopt}, \Pkg{PyManopt} \cite{townsend2016pymanopt}, \Pkg{ROPTLIB} \cite{huang2018roptlib}), or specialized for training neural networks with manifold constraints (e.g., \Pkg{Geoopt} \cite{geoopt2020kochurov}, \Pkg{McTorch} \cite{meghwanshi2018mctorch}, \Pkg{Rieoptax} \cite{rieoptax2022rieoptax}).     
	In these packages, the manifold is described by the constraints, together with the geometrical materials from differential geometry, including retractions and their inverses, vector transports, etc. Some packages, including \Pkg{Geomstats} \cite{geomstats2021}, even require the exact formulation of exponential mappings and logarithm mappings. Unfortunately, the task of computing these geometrical materials is usually non-trivial. Although these existing Riemannian optimization packages provide several predefined manifolds based on various excellent works on determining the geometrical materials for some important manifolds  \cite{edelman1998geometry,nickel2018learning,bendokat2020grassmann,gao2021riemannian}, it is still non-trivial to add new manifold constraints for users and developers, especially when they are nonexperts in differential geometry.

	As manifold constraints are characterized by geometrical materials, Riemannian optimization solvers are developed from unconstrained optimization ones by incorporating these geometrical materials. To extend a large collection of unconstrained optimization solvers to its Riemannian versions, the following three steps are indispensable: (i) replace Euclidean differentials by Riemannian differentials; (ii)  invoke retractions to keep the feasibility of the iterates;  (iii)  employ vector transports to move vectors among the tangent spaces of the manifolds. Therefore, building up new Riemannian optimization algorithms by borrowing the advanced features of state-of-the-art unconstrained optimization algorithms is challenging in general. Existing Riemannian optimization packages only provide limited Riemannian optimization algorithms. For example, compared with over $30$ efficient optimizers from \Pkg{PyTorch} and \Pkg{PyTorch-optimizer} packages for unconstrained optimization,  \Pkg{Geoopt} only provides Riemannian Adam and Riemannian SGD, and \Pkg{McTorch} is only integrated with Riemannian AdaDelta and Riemannian SGD.

	Additionally, the package \Pkg{GeoTorch} \cite{lezcano2019trivializations} is developed based on the trivialization approach for training neural networks with manifold constraints by \Pkg{PyTorch}. In \Pkg{GeoTorch}, the manifold is characterized by a smooth surjective mapping (trivialization mapping) $\psi: \bb{R}^N \to \M$ for some constant $N$. Then \Pkg{GeoTorch} transforms \ref{Prob_Ori} into the following unconstrained optimization problem,
	\begin{equation*}
		\min_{u \in \bb{R}^N} \quad f(\psi(u)).
	\end{equation*}
	However, determining the trivialization mapping $\psi$ is usually as challenging as determining the geometrical materials of $\M$. Until now, the formulations of $\psi$ are only determined for spheres and Stiefel manifolds. 
	
	On the other hand, although \Pkg{GeoTorch} enables the direct implementation of unconstrained optimization solvers to \ref{Prob_Ori} for spheres and Stiefel manifolds,  the expression of trivialization mappings in \Pkg{GeoTorch} are usually too complicated to be computed efficiently. As mentioned in \cite{wang2020orthogonal,xiao2021solving,hu2022constraint,ablin2022fast}, computing the mapping $\psi$ provided in \Pkg{GeoTorch} is usually costly, and could become the major computational bottleneck for the optimizers. 
	For example, when $\M$ is chosen as the Stiefel manifold $\ca{S}_{m,s}$ embedded in $\bb{R}^{m\times s}$, $\psi$ is chosen as the Cayley transform or matrix exponential of an $n\times n$ skew-symmetric matrix in \Pkg{GeoTorch}. Then computing $\psi$ and its differentials are extremely expensive, leading to its inferior performance when compared with other optimization approaches \cite{wang2020orthogonal,ablin2022fast,hu2022constraint}.

	Apart from the trivialization approach, \cite{ablin2022fast} introduces the landing method for minimizing continuously differentiable functions over the Stiefel manifold. Then following the framework of the landing method, \cite{gao2022optimization}  incorporates the canonical metric to the landing method \cite{gao2022optimization}, and \cite{ablin2023infeasible} extends the landing method to solve stochastic optimization over the Stiefel manifold. Despite these advancements, these variants of the landing method are primarily developed based on the vanilla (stochastic) gradient descent method and do not integrate widely-used acceleration techniques (e.g., heavy-ball momentum, Nesterov momentum, adaptive stepsizes). Furthermore, the existing approaches derived from the landing method are mainly restricted to problems on the Stiefel manifold. The challenge of extending the landing method to address the broader class of problems represented by \eqref{Prob_Ori} remains an open question.

	Furthermore, by regarding \eqref{Prob_Ori} as a constrained optimization problem, some existing penalty function approaches can be applied to solve \eqref{Prob_Ori}. The nonsmooth penalty function approach \cite{nocedal2006numerical} is a simple and fundamental approach for constrained optimization, where the penalty function can be expressed as,
	\begin{equation*}
		\psi_{NS}(x) := f(x) + \beta \norm{c(x)}.
	\end{equation*}
	However, the penalty function $\psi_{NS}$ is nonsmooth even if $f$ and $c$ are continuously differentiable. Therefore, the minimization of $\psi_{NS}$ is challenging in practice, as a wide range of efficient methods for smooth unconstrained optimization (e.g., quasi-Newton method, conjugate gradient method \cite{fletcher2002nonlinear}, trust-region method \cite{powell1991trust}) require the differentiability of the objective functions. 
	
	On the other hand, several existing works \cite{estrin2020implementing,goyens2022computing} have developed exact penalty function approaches based on the Fletcher's penalty function \cite{fletcher1964function} defined by
	\begin{equation*}
		\psi_{FL}(x) := f(x) - \inner{ \Jc(x)^\dagger \nabla f(x), c(x) } + \beta \norm{c(x)}^2.
	\end{equation*}
	Here $\Jc(x)^\dagger$ is the pseudo-inverse of $\Jc(x)$. However, the evaluation of
	%%the objective function of 
	$\psi_{FL}(x)$ requires the computation of $\nabla f(x)$, and computing the differentials of $\psi_{FL}$ requires the computation of higher-order differentials of the objective function $f$. As a result,  existing approaches based on Fletcher's penalty function, such as inexact steepest descent methods \cite{estrin2020implementing}, inexact Newton methods \cite{toint1981towards,steihaug1983conjugate,zavala2014scalable}, and inexact quasi-Newton methods \cite{estrin2020implementing}, are usually combined with certain approximation strategies to waive the computation of higher-order derivatives. Hence, various existing unconstrained optimization approaches are not compatible with the Fletcher's penalty function framework.

	Given the challenges in developing efficient Riemannian optimization approaches based on existing Riemannian optimization packages, we are driven to ask the following question:
	\begin{quote}
		\it
		Can we develop an optimization package that enables direct implementations of unconstrained optimization solvers for a Riemannian optimization problem, without the need to compute any geometrical materials of the corresponding manifold?
	\end{quote}

	\subsection{Developing optimization packages from the constraint dissolving approach}	
	
	To directly apply unconstrained optimization approaches for solving Riemannian optimization problems, \cite{xiao2022constraint} proposed the {\it constraint dissolving} approach for \ref{Prob_Ori}, under the assumption that linear independence constraint qualification (LICQ) holds at any feasible point of $\M$. The constraint dissolving approach transforms \ref{Prob_Ori} into the unconstrained minimization of the following constraint dissolving function (\ref{Prob_Pen}),
	\begin{equation}
		\tag{CDF}
		\label{Prob_Pen}
		h(x) := f(\A(x)) + \frac{\beta}{2} \norm{c(x)}^2.
	\end{equation} 
	Here  $\A : \bb{R}^n \to \bb{R}^n$ is called the constraint dissolving mapping, which should satisfy the following assumptions.
	\begin{assumpt}{\bf Blanket assumptions on $\A$}
		\label{Assumption_2}
		
		\begin{itemize}
			\item $\ca{A}$ is locally Lipschitz smooth in $\bb{R}^n$;
			\item $\A(x) = x$ holds for any $x \in \M$;
			\item The Jacobian of $c(\A(x))$ equals to $0$ for any $x \in \M$. That is, $\Ja(x) \Jc(x) = 0$  holds for any $x \in \M$, where $\Ja(x) \in \bb{R}^{n\times n}$ denotes the transposed Jacobian of $\A$ at $x$.
		\end{itemize}
	\end{assumpt}
	The details on how to construct suitable constraint dissolving mappings $\A$
	for $\M$ are presented in Section 4.1. 
	
	When $f$ is assumed to be locally Lipschitz smooth, \cite{xiao2022constraint} proves that \ref{Prob_Ori} and \ref{Prob_Pen} have the same first-order stationary points, second-order stationary points, local minimizers, and {\L}ojasiewicz exponents in a neighborhood of $\M$. More importantly, \cite{xiao2022constraint} shows that constructing CDF is completely independent of any geometrical material of $\M$. Therefore, we can develop various constraint dissolving approaches to solve optimization problems over a broad class of manifolds embedded in $\bb{R}^n$, without any prior knowledge of their geometrical properties. Note that \cite{xiao2022constraint} analyze the relationships between \ref{Prob_Ori} and \ref{Prob_Pen} when LICQ holds everywhere on $\M$, which is a more stringent assumption than Assumption \ref{Assumption_1}. Interested readers could further refer to \cite{xiao2021solving,xiao2020class,xiao2020l21} for developing constraint dissolving approaches and efficient algorithms for smooth optimization with orthogonality constraints.
	
	For nonsmooth scenarios, \cite{hu2022constraint} extends the constraint dissolving approaches for nonsmooth optimization over the Stiefel manifold, which enables the direct implementation of various existing unconstrained optimizers from \Pkg{PyTorch} for training DNNs with orthogonally constrained weights. In another line of work, \cite{hu2022improved} shows that the feasible region of any nonconvex-strongly-convex bilevel optimization problem is a Riemannian manifold, and proposes a novel constraint dissolving approach for these bilevel optimization problems. Moreover, \cite{hu2022improved} develops a general framework for designing subgradient methods and proves their convergence properties for bilevel optimization problems.

	However, these existing approaches either focus on specialized manifolds (e.g., Stiefel manifold), or require the global LICQ condition for the constraints $c(x)$. The development of constraint dissolving functions for \eqref{Prob_Ori} under more relaxed constraint qualifications has not been adequately addressed in the existing literature. More importantly, there is  currently no software that is developed based on the constraint dissolving approach for solving Riemannian optimization problems. On the other hand, the penalty parameter $\beta$ plays an important role in the exactness of the constraint dissolving function $h(x)$. Although \cite{xiao2022constraint} suggests a threshold value for $\beta$ to guarantee the equivalence between \ref{Prob_Ori} and \ref{Prob_Pen} in a neighborhood of the manifold $\M$, such a  threshold may be difficult to estimate. How to practically choose the penalty parameter $\beta$ for \ref{Prob_Pen} remains unexplored. In this paper, we will propose an easy-to-implement scheme for estimating the threshold value for the penalty parameter $\beta$.

	\subsection{Contributions}

	In this paper, we present a Python package called \CDOpt, which is developed based on the constraint dissolving approaches for both standard Riemannian optimization problems and training neural networks with manifold constraints. Different from existing Riemannian optimization packages, \CDOpt simultaneously achieves the following goals,
	\begin{itemize}
		\item {\bf Dissolved constraints:}  % Constraint dissolving
		Developed based on the constraint dissolving approaches, \CDOpt first transforms Riemannian optimization problems into unconstrained optimization problems. Therefore, we can utilize various highly efficient solvers for unconstrained optimization, and directly apply them to solve Riemannian optimization problems. Benefiting from the rich expertise gained over the past decades for unconstrained optimization, \CDOpt is very efficient and naturally avoids the needs and difficulties in extending the unconstrained optimization solvers to their Riemannian versions.
		\item {\bf High compatibility:}
		\CDOpt has high compatibility with various existing numerical backends, including \Pkg{NumPy}  \cite{numpy2020array}, SciPy \cite{scipy2020SciPyNMeth}, \Pkg{Autograd} \cite{autograd2020}, \Pkg{PyTorch} \cite{pytorch2019pytorch}, \Pkg{JAX} \cite{jax2018github} and Flax \cite{flax2020github}. Users can directly apply the advanced features of these packages to accelerate the optimization, including automatic differentiation, GPU/TPU supports, just-in-time (JIT) compilation, etc.
		\item {\bf Customized manifolds:}
		\CDOpt dissolves manifold constraints without involving any geometrical material of the manifold in question. Therefore, users can directly define various Riemannian manifolds in \CDOpt through their constraint expressions $c(x)$.
		\item {\bf Plug-in neural layers:} 
		\CDOpt provides various plug-in neural layers for PyTorch and \Pkg{Flax} packages. With minor changes to the standard \Pkg{PyTorch} / \Pkg{Flax} codes, users can easily build and train neural networks with various manifold constraints based on \Pkg{PyTorch} and \Pkg{JAX}. 
	\end{itemize}
	
	These features make \CDOpt an easy-to-use package complementing the existing Riemannian optimization packages, in the sense that \CDOpt provides an alternative approach for  Riemannian optimization by direct implementation of unconstrained solvers.
	Furthermore, we also present improved theoretical analysis on the relationships between \ref{Prob_Pen} and \ref{Prob_Ori}, 
	\begin{itemize}
		\item {\bf Relaxed constraint qualification:} 
		We establish the equivalence between \ref{Prob_Ori} and \ref{Prob_Pen} under the RCRCQ condition, in the sense that \ref{Prob_Ori} and \ref{Prob_Pen} admit the same first-order and second-order stationary points in a neighborhood of $\M$. The involvement of the RCRCQ condition is weaker than the LICQ condition assumed in existing works \cite{xiao2021solving,xiao2022constraint,hu2022constraint,hu2022improved}. Therefore, the constraint dissolving approaches can be applied to the optimization problems over a broader class of embedded Riemannian submanifolds, including the Stiefel manifold with range constraints (i.e., the manifold studied in \cite{huang2022riemannian}). 
		\item {\bf Easy-to-implement scheme for tuning the penalty parameter:}
		We introduce a novel scheme for estimating the penalty parameter $\beta$ for \ref{Prob_Pen} to satisfy the exactness property in a neighborhood of $\M$. Compared with the threshold values presented in \cite{xiao2022constraint}, which necessitate the determination of Lipschtz constants for $f$, $\A$ and $c$, our proposed scheme only involves the computation of $f$, $\A$, $c$ and their derivatives. In the \CDOpt package, we employ the Monte Carlo sampling method to estimate the threshold value for the penalty parameter $\beta$ based on our proposed scheme. As a result, \CDOpt supports the automatic selection of $\beta$, hence overcoming the difficulties associated with choosing appropriate penalty parameters in existing constraint dissolving approaches.
		
	\end{itemize}

	\subsection{Organizations}
	The rest of this paper is organized as follows. Section 2 presents some basic notations, definitions, and constants that are necessary for the theoretical analysis in later parts. We establish the theoretical properties of \ref{Prob_Pen} and propose a practical scheme for choosing the penalty parameter in Section 3.   
	Section 4 illustrates the structure of \CDOpt and describes its main modules. In Section 5, we present a brief comparison between \CDOpt and other existing Python Riemannian optimization packages. Section 4 exhibits several illustrative examples on applying \CDOpt to solve standard Riemannian optimization problems and train neural networks with manifold constraints. Some examples on how to use \CDOpt to solve Riemannian optimization problems are presented in Section 6.  In Section 7, we present preliminary numerical experiments to illustrate that \CDOpt enables efficient and direct implementation of various existing unconstrained solvers to solve \ref{Prob_Ori}. We draw a brief conclusion in the last section.

	\section{Notations, definitions, and constants}
	\subsection{Notations}
	Let $\mathrm{range}(A)$ be the subspace spanned by the column vectors of matrix $A$, and $\norm{\cdot}$ represents the $\ell_2$-norm of a vector or an operator.
	The notations $\mathrm{diag}(A)$ and $\mathrm{Diag}(x)$
	stand for the vector formed by the diagonal entries of a matrix $A$,
	and the diagonal matrix with the entries of $x\in\bb{R}^n$ as its diagonal, respectively. 
	%$X^\dag$ refers to the pseudo-inverse of $X$.
	We denote the smallest and largest eigenvalues of $A$ by $\lambda_{\mathrm{min}}(A)$ and $\lambda_{\max}(A)$, respectively. Besides, $\sigma_{s}(A)$ refers to the $s$-th largest singular value of $A$ while  $\sigma_{\min}(A)$ refers to the smallest singular value of matrix $A$. Furthermore, for any matrix $A \in \bb{R}^{n\times p}$, 
	the pseudo-inverse of $A$ is denoted by $A^\dagger \in \bb{R}^{p\times n}$, which satisfies $AA^\dagger A = A$, $A^\dagger AA^\dagger = A^\dagger$, and both $A^{\dagger} A$ and $A A^{\dagger}$ are symmetric \cite{GolubMatrix}. 
	
	In this paper, the Riemannian metric for $\M$ is chosen as the Euclidean metric in $\bb{R}^n$. For any $x\in\M$,  we denote the Riemannian gradient and Riemannian Hessian of $f$ at $x$ as $\grad f(x)$ and $\hess f(x)$, respectively.
	
	For any $x \in \bb{R}^n$, we define the projection from $x \in \bb{R}^n$ to $\M$ as 
	\begin{equation*}
		\mathrm{proj}(x, \M) := \mathop{\arg\min}_{y \in \M} ~ \norm{x-y}. 
	\end{equation*} 
	It is worth mentioning that the optimality condition of the above problem leads to the fact that $x - w \in \mathrm{range}(\Jc(w))$ for any $w \in \mathrm{proj}(x, \M)$.
	Furthermore, $\mathrm{dist}(x, \M)$ refers to the distance between $x$ and $\M$, i.e. 
	\begin{equation*}
		\mathrm{dist}(x, \M) := \mathop{\inf}_{y \in \M} ~ \norm{x-y}.
	\end{equation*}

	We denote the  transposed Jacobian of the mapping $\A$ by $\Ja(x) \in \bb{R}^{n\times n}$. Let $c_i$ and $\A_{i}$ be the $i$-th coordinate of the mapping $c$ and $\A$ respectively, then $\Jc$ and $\Ja$ can be expressed as
	\begin{equation*}
		\Jc(x) := \left[  \begin{smallmatrix}
			\frac{\partial c_1(x)}{\partial x_1} & \cdots & \frac{\partial c_p(x)}{\partial x_1} \\
			\vdots & \ddots & \vdots \\
			\frac{\partial c_1(x)}{\partial x_n} & \cdots & \frac{\partial c_p(x)}{\partial x_n} \\
		\end{smallmatrix}\right] \in \bb{R}^{n\times p},
		\quad \text{and} ~   
		\Ja(x) := \left[  \begin{smallmatrix}
			\frac{\partial \A_1(x)}{\partial x_1} & \cdots & \frac{\partial \A_n(x)}{\partial x_1} \\
			\vdots & \ddots & \vdots \\
			\frac{\partial \A_1(x)}{\partial x_n} & \cdots & \frac{\partial \A_n(x)}{\partial x_n} \\
		\end{smallmatrix}\right] \in \bb{R}^{n\times n}. 
	\end{equation*}
	Besides, $\DJa(x): d \mapsto \DJa(x)[d]$ denotes the second-order derivative of the mapping $\A$, which can be regarded as a linear mapping from $\bb{R}^n$ to $\bb{R}^{n\times n}$ that satisfies $\DJa(x)[d] = \lim_{t \to 0} \frac{1}{t}(\Ja(x + td) - \Ja(x) )$. 
	Similarly,  $\DJc(x)$ refers to the second-order derivative of the mapping $c$ that satisfies $\DJc(x)[d] = \lim_{t \to 0} \frac{1}{t}(\Jc(x + td) - \Jc(x) ) \in \bb{R}^{n\times p}$. Additionally,  we set $$\A^{k}(x) := \underbrace{\A(\A(\cdots\A(x)\cdots))}_{k \text{ times}},$$
	for $k \geq 1$, and define $\A^0(x) := x$, $\A^{\infty}(x):= \lim\limits_{k \to +\infty} \A^k(x)$. Furthermore, we denote $g(x) := f(\A(x))$ and use $\nabla f(\A(x))$ to denote $\nabla f(z)\large|_{z = \A(x)}$ in the rest of this paper.

	\subsection{Definitions}
	We first state the first-order optimality condition of \ref{Prob_Ori} as follows.
	\begin{defin}[\cite{nocedal2006numerical}]\label{Defin_FOSP}
		Given $x \in \bb{R}^n$, we say $x$ is a first-order stationary point of \ref{Prob_Ori} if there exists $\tilde{\lambda} \in \bb{R}^p$ that satisfies  
		\begin{equation}
			\label{Eq_Defin_FOSP}
			\left\{\begin{aligned}
				\nabla f(x) -  \sum_{i=1}^p\tilde{\lambda}_i \nabla c_i(x) &= 0,\\
				c(x) &= 0.
			\end{aligned}\right.
		\end{equation}
	\end{defin}

	For any given $x \in \M$, 
	we define $\lambda(x)$ as $$\lambda(x) := \Jc(x)^\dagger \nabla f(x)\in\argmin\limits_{\lambda\in\bb{R}^p} \norm{\nabla f(x) - \Jc(x) \lambda}.$$
	Then it can be easily verified that  
	\begin{equation}\label{add:2}
		\nabla f(x) = \Jc(x) \lambda(x), 
	\end{equation}
	whenever $x$ is a first-order stationary point of \ref{Prob_Ori}.

	As the Riemannian metric on $\M$ is fixed as the Euclidean metric in $\bb{R}^n$, the following propositions present the closed-form expressions for the tangent space, normal space, Riemannian gradient $\grad f(x)$ and Riemannian Hessian $\hess f(x)$ for any $x \in \M$. The proofs of these propositions directly follow \cite{Absil2009optimization,boumal2020introduction}, and hence are omitted for simplicity. 
	\begin{prop}
		Given $x \in \M$,  the tangent space of $\M$ at $x$ can be expressed as $\Tx = \{ d \in \bb{R}^{n}: d\tp \Jc(x) = 0 \}$, while the normal space of $\M$ at $x$ can be expressed as $\Nx = \mathrm{Range}(\Jc(x))$. 
	\end{prop}
	
	\begin{defin}\label{Defin_SOSP}
		Given $x \in \bb{R}^n$, we say $x$ is a second-order stationary point of \ref{Prob_Ori} if $x$ is a first-order stationary point of \ref{Prob_Ori} and for any $d \in \Tx$, it holds that 
		\begin{equation}
			d\tp \left(\nabla^2 f(x) - \sum_{i=1}^p \lambda_i(x) \nabla^2 c_i(x)\right)d \geq 0.
		\end{equation}
	\end{defin}
	
	\begin{prop}
		\label{Prop_FOSP_Rie}
		
		The Riemannian gradient of $f$ at $x$ can be expressed as  
		\begin{equation}\label{add:4}
			\grad f(x) = \nabla f(x) - \Jc(x)\lambda(x).
		\end{equation}
		Moreover, $\hess f(x)$ can be expressed as the following self-adjoint linear transform $\hess f(x) : \Tx \to \Tx$ such that
		\begin{equation*}
			d\tp \hess f(x) d =  d\tp \left(\nabla^2 f(x) - \sum_{i=1}^p \lambda_i(x) \nabla^2 c_i(x)\right)d, \qquad \text{for any } d \in \Tx. 
		\end{equation*}
	\end{prop}

	Given $x \in \M$,  the smallest eigenvalue of $\hess f(x)$ is defined as $$\lambda_{\min}(\hess f(x)) := \min_{d \in \Tx, \norm{d} = 1} ~d\tp \hess f(x) d.$$
	Let $U_x$ be a matrix whose columns form an orthonormal basis of $\Tx$, we define the projected Hessian of \ref{Prob_Ori} at $x$ as 
	\begin{equation}
		\label{Eq_defin_HX}
		\ca{H}(x):=  U_x\tp \left(\nabla^2 f(x) - \sum_{i=1}^p \lambda_i(x) \nabla^2 c_i(x)\right)U_x,
	\end{equation}
	The following proposition characterizes the relationship between $\ca{H}(x)$ and $\hess f(x)$. 
	
	\begin{prop}
		\label{Prop_SOSP_Rie}
		Given any $x \in \M$, suppose $x$ is a first-order stationary point of \ref{Prob_Ori}, then $\ca{H}(x)$ and $\hess f(x)$ have the same eigenvalues. Moreover, we have that $\lambda_{\min}(\hess f(x)) = \lambda_{\min}(\ca{H}(x))$ and $x$ is a second-order stationary point of \ref{Prob_Ori} if and only if $\ca{H}(x) \succeq 0$. 
	\end{prop}
	The proof of the above proposition directly follows from Proposition \ref{Prop_FOSP_Rie} and \cite{Absil2009optimization}, and hence we omit it for simplicity.

	\begin{defin}
		Given $x \in \bb{R}^n$, we say $x$ is a first-order stationary point of \ref{Prob_Pen} if
		\begin{equation}
			\nabla h(x) = 0.
		\end{equation}
		Moreover, when $h$ is twice differentiable, we say $x \in \bb{R}^n$ is a second-order stationary point of $h$ if $x$ is a first-order stationary point of $h$ and satisfies
		\begin{equation}
			\nabla^2 h(x) \succeq 0. 
		\end{equation}
	\end{defin}

	\subsection{Constants}
	As illustrated in Assumption \ref{Assumption_1}, the rank of $\Jc(x)$ equals $p'$ for any $x$ in $\X$, which is a neighborhood of $\M$. Then for any given $x \in \M$, we define the positive scalar $\rho_x \leq 1$ as
	\begin{equation*}
		\rho_x := \mathop{\arg\max}_{0 < \rho \leq 1, ~\ca{B}_{x, \rho} \subset \X}~ \rho \qquad \text{s.t.} ~ \inf \left\{\sigma_{p'}(\Jc(y)): y \in \ca{B}_{x, \rho}\right\} \geq \frac{1}{2} \sigma_{p'}(\Jc(x)). 
	\end{equation*}
	Based on the definition of $\rho_x$, we can define the set $\Theta_x:= \{ y \in \bb{R}^n: \norm{y-x} \leq \rho_x \}$ and define several constants as follows:
	\begin{itemize}
		\item $\Lsc  := \sigma_{p'}(\Jc(x))$; 
		\item $\Lf  := \sup_{y \in \Theta_x}~ \norm{\nabla f(\A(y))}$;
		\item $\Mc  := \sup_{y \in \Theta_x}   \norm{\Jc(y)} $;
		\item $\Ma  := \sup_{y \in \Theta_x}  \norm{\Ja(y)}$;
		\item $\Lg  := \sup_{y, z \in \Theta_x, y\neq z} ~   \frac{\norm{\nabla g(y) - \nabla g(z)}}{\norm{y-z}}$;
		\item $\Lc := \sup_{y, z \in \Theta_x, y\neq z} \frac{\norm{\Jc(y) - \Jc(z) }}{\norm{y-z}} $;
		\item $\La := \sup_{y, z \in \Theta_x, y\neq z} \frac{\norm{\Ja(y) - \Ja(z) }}{\norm{y-z}} $;
		\item $\Lac := \sup_{y, z \in \Theta_x, y\neq z} \frac{\norm{ \Ja(y)\Jc(\A(y)) - \Ja(z)\Jc(\A(z)) }}{\norm{y-z}} $.
	\end{itemize}
	
	It should be noted that the constants mentioned above are primarily utilized for theoretical analysis to establish the relationships between \ref{Prob_Ori} and \ref{Prob_Pen}.  
	In particular,  the threshold value given in Definition \ref{Cond_beta_esti} 
	in Section 3 for $\beta$ to ensure the equivalence between  \ref{Prob_Ori} and \ref{Prob_Pen}, 
	is free of the constants mentioned above. 
	Consequently, it is not necessary to compute these aforementioned constants in the practical implementations of the \CDOpt package.
	
	Based on these constants, we further set 
	\begin{equation*}
		\varepsilon_x := \min \left\{ \frac{\rho_{x}}{2}, \frac{\Lsc}{32\Lc(\Ma+1)}, \frac{\Lsc ^2}{8(\Lac + \Lc )\Mc }, \frac{\Lsc}{5(\Lac  + \La\Mc) }   \right\},
	\end{equation*}
	and define the following sets:
	\begin{itemize}
		
		\item $\Omegax{x} := \left\{ y \in \bb{R}^n: \norm{y - x}  \leq \varepsilon_x \right\}$;
		\item $\BOmegax{x} := \left\{ y \in \bb{R}^n: \norm{y - x}  \leq \frac{\Lsc \varepsilon_x}{4\Mc(\Ma +1) + \Lsc}  \right\}$; 
		\item $\Omega:= \bigcup_{x \in \M} \Omegax{x}$;
		\item $\bar{\Omega}:= \bigcup_{x \in \M} \BOmegax{x}$.
	\end{itemize}
	It is worth mentioning that Assumption \ref{Assumption_1} guarantees that $\Lsc > 0$ for any given $x \in \M$, which implies that $\varepsilon_x > 0$. On the other hand, we can conclude that $\BOmegax{x} \subset \Omegax{x} \subset \Theta_x$ holds for any given $x \in \M$, and $\M$ lies in the interior of $\bar{\Omega}$.

	Finally, the following assumption is needed in analyzing the second-order stationarity of \ref{Prob_Pen}.
	\begin{assumpt}{\bf Assumption on second-order differentiability}
		\label{Assumption_4}
		
		\begin{itemize}
			\item $f$, $c$ and $\A$ are twice differentiable in $\bb{R}^n$.
		\end{itemize}
	\end{assumpt}

	\section{Theoretical Properties}
	In this section, we present the theoretical analysis on the relationships between the stationary point of \ref{Prob_Ori}  and those of \ref{Prob_Pen} under the RCRCQ condition. We first establish some basic properties of the constraint dissolving mapping $\A$ in Section \ref{Subsection_Theo_properties_A}. Based on these theoretical properties of $\A$, we propose a practical scheme on choosing the threshold value for penalty parameter $\beta$ for \ref{Prob_Pen}, and discuss how to compute the threshold value in practice. Finally, we prove that for any penalty $\beta$ larger than the proposed threshold value, \ref{Prob_Ori} and \ref{Prob_Pen} have the same first-order and second-order stationary points in a neighborhood of $\M$. 
	
	\subsection{Theoretical properties of $\A$}
	\label{Subsection_Theo_properties_A}
	In this subsection, we aim to establish the same theoretical results as \cite[Section 3.1]{xiao2022constraint} under Assumption \ref{Assumption_1}, where we only assume that RCRCQ holds for every feasible point of \ref{Prob_Ori}. We start by evaluating the relationships among $\norm{c(x)}$, $\norm{c(\A(x))}$ and $\mathrm{dist}(x, \M)$ in the following lemmas.
	\begin{lem}
		\label{Le_bound_cx}
		For any given $x \in \M$, the following inequalities hold for any $\y \in \Omegax{x}$,
		\begin{equation}
			\frac{1}{\Mc } \norm{c(\y)} \leq \mathrm{dist}(\y, \M) \leq \frac{2}{\Lsc } \norm{c(\y)}.
		\end{equation}
	\end{lem}
	\begin{proof}
		Let $z \in \mathrm{proj}(y, \M)$, then from the definition of $z$ we can conclude that $\norm{z-y}\leq \norm{y-x}$. Thus $\norm{z-x} \leq \norm{z-y} + \norm{y-x} \leq 2\varepsilon_x \leq \rho_x$, and hence $z \in \Theta_x$. 
		By the mean-value theorem, for any fixed $\nu \in \bb{R}^p$ there exists a point $\xi_\nu \in  \bb{R}^n$ that is a convex combination of $\y$ and $z$ such that $ \nu\tp c(y) = (y-x)\tp\Jc(\xi_\nu)\nu $. By the convexity of $\Theta_x$, $\xi_\nu \in \Theta_x$ holds for any $\nu \in \bb{R}^p$. Therefore, we get
		\begin{equation*}
			\begin{aligned}
				\norm{c(\y)} ={}& \sup_{\nu \in \bb{R}^p, \norm{\nu} = 1} \nu\tp c(y) = \sup_{\nu \in \bb{R}^p, \norm{\nu} = 1} (y-x)\tp\Jc(\xi_\nu)\nu  \\
				\leq {}&
				\sup_{\xi_\nu \in \Theta_x} 
				\norm{\Jc(\xi_\nu)} \norm{\y - z} \leq \Mc  \mathrm{dist}(\y, \M).
			\end{aligned}
		\end{equation*}
		Moreover, it follows from the definition of $z$ that $y-z \in \mathrm{range}(\Jc(z))$.  As a result, let $\tilde{\nu} = \frac{\Jc(z)\tp(y-z)}{\norm{\Jc(z)\tp(y-z)}}$, we have
		\begin{equation*}
			\begin{aligned}
				&\norm{c(\y)}= \sup_{\nu \in \bb{R}^p, \norm{\nu} = 1} (y-z)\tp\Jc(\xi_\nu)\nu  \geq (y-z)\tp\Jc(\xi_{\tilde{\nu}})\tilde{\nu}  \\
				={}&  (y-z)\tp\Jc(z)\tilde{\nu} + (y-z)\tp\left(\Jc(z) - \Jc(\xi_{\tilde{\nu}})\right)\tilde{\nu}  \\
				={}& \norm{ \Jc(z)\tp (y-z)} + (y-z)\tp\left(\Jc(z) - \Jc(\xi_{\tilde{\nu}})\right)\tilde{\nu} \\
				\geq{}& \norm{ \Jc(z)\tp (y-z)} - \Lc \norm{y-z}^2 \\
				\geq{}&  (\Lsc - \varepsilon_x \Lc )\mathrm{dist}(\y, \M)  \geq  \frac{\Lsc }{2} \mathrm{dist}(\y, \M).
			\end{aligned}
		\end{equation*}
	\end{proof}

	\begin{lem}
		\label{Le_Ax_c}
		For any given $x \in \M$, it holds that
		\begin{equation}
			\norm{\A(\y) - \y }  \leq \frac{2(\Ma +1)}{\Lsc }\norm{c(\y)}, \qquad \text{ for any $\y \in \Omegax{x}$}.
		\end{equation}
	\end{lem}
	\begin{proof}
		For any given $\y \in \Omegax{x}$, we choose $z \in \mathrm{proj}(y, \M)$. Then we can conclude that $z \in \Theta_x$. Furthermore, from the Lipschitz continuity of $\A$ and the fact that $\A(z) - z = 0$,  it holds that
		\begin{equation}
			\begin{aligned}
				&\norm{\A(\y) - \y} = \norm{(\A(\y) - \y)  - (\A(z) - z)}  \leq  (\Ma +1)\mathrm{dist}(y, \M) \\
				\leq{}& \frac{2(\Ma +1)}{\Lsc }\norm{c(\y)},
			\end{aligned}
		\end{equation}
		where the last inequality follows from Lemma \ref{Le_bound_cx}.
	\end{proof}

	\begin{lem}
		\label{Le_A_secondorder_descrease}
		For any given $x \in \M$,  it holds that 
		\begin{equation}
			\norm{c(\A(\y))} \leq \frac{4\Lac }{\Lsc ^2}\norm{c(\y)}^2, \qquad \text{for any $\y \in \Omegax{x}$}.
		\end{equation}
	\end{lem}
	\begin{proof}
		
		For any given $\y \in \Omegax{x}$, we choose $z \in \mathrm{proj}(y, \M)$. It holds that $z \in \Theta_x$. By the mean-value theorem, for any $\nu \in \bb{R}^p$, there 
		exists $t_\nu \in [0,1]$ and $\xi_\nu = t_\nu \y + (1-t_\nu)z$ such that
		\begin{equation}
			\nu\tp  c(\A(\y)) = \nu\tp \left(\Ja(\xi_\nu)\Jc(\A(\xi_\nu))\right)\tp (\y - z) .
		\end{equation}
		The convexity of $\Theta_x$ ensures that $\xi_\nu \in \Theta_x$ holds for any $\nu \in \bb{R}^p$. 
		Therefore, from the definition of $\Lac $ and $\Omegax{x}$, we get
		\begin{equation}
			\begin{aligned}
				&\norm{c(\A(\y))} = \sup_{\nu \in \bb{R}^p, \norm{\nu} = 1}  \nu\tp  c(\A(\y)) \\
				={}& \sup_{\nu \in \bb{R}^p, \norm{\nu} = 1}  
				\nu\tp \left(\Ja(\xi_\nu)\Jc(\A(\xi_\nu))\right)\tp (\y - z) 
				\\
				\leq{}& \sup_{\nu \in \bb{R}^p, \norm{\nu} = 1} \norm{\left(\Ja(\xi_\nu)\Jc(\A(\xi_\nu))\right)\tp (\y - z)} 
				\\
				\leq{}&   \sup_{\nu \in \bb{R}^p, \norm{\nu} = 1} \norm{\Ja(\xi_\nu)\Jc(\A(\xi_\nu))}\mathrm{dist}(\y, \M)\\
				\leq{}& \Lac  \sup_{\nu \in \bb{R}^p, \norm{\nu} = 1} \norm{\xi_\nu - z} \mathrm{dist}(\y, \M) 
				\leq \Lac  \mathrm{dist}(\y, \M) ^2 \\
				\leq{}& \frac{4\Lac }{\Lsc ^2}\norm{c(\y)}^2,
			\end{aligned}
		\end{equation}
		where the last inequality follows from Lemma \ref{Le_bound_cx}. 
	\end{proof}
	
	For any given $x \in \M$ and $y \in \Omegax{x}$, Lemma \ref{Le_A_secondorder_descrease} illustrates that the operator $\A$ can reduce the feasibility violation of $y$ quadratically when $y$ is sufficiently close to $\M$. Moreover, the following lemma illustrates that although $\Jc(x)$ may be singular for some $x \in \M$,  $\norm{\Jc(y)c(y)}$ is still proportional to $\norm{c(y)}$ for any $y \in \Omegax{x}$. 
	
	\begin{lem}
		\label{Le_Jccx_cx_relationship}
		For any given $x \in \M$ and any $y \in \Omegax{x}$,  it holds that 
		\begin{equation}
			\norm{\Jc(y)c(y)} \geq \frac{\Lsc}{2} \norm{c(y)}. 
		\end{equation}
	\end{lem}
	\begin{proof}
		For any given $y \in \Omegax{x}$, we choose $z \in \mathrm{proj}(y, \M)$. Therefore, for any $\nu \in \bb{R}^p$, there exists $t_\nu \in [0,1]$ and $\xi_\nu = t_\nu \y + (1-t_\nu)z$ such that
		\begin{equation}
			\nu\tp c(y) = \nu\tp \Jc(\xi_\nu)\tp (y-z) \geq  \nu\tp \Jc(y)\tp(y-z) - \Lc \norm{y-z}^2.
		\end{equation}
		Then we can conclude that $\norm{c(y) - \Jc(y)\tp(y-z)} \leq \Lc \norm{y-z}^2$, which leads to 
		\begin{equation}
			\begin{aligned}
				&\norm{\Jc(y)c(y)} \geq \norm{\Jc(y)\Jc(y)\tp (y-z)} - \Lc\Mc\norm{y-z}^2 \\
				\geq{}& \Lsc \norm{\Jc(y)\tp (y-z)} - \Lc\Mc\norm{y-z}^2\\
				\geq{}& \Lsc \norm{c(y)} - (\Lc\Mc + \Lsc\Lc) \norm{y-z}^2\\
				\overset{(i)}{=}{}{}& \Lsc \norm{c(y)} - (\Lc\Mc + \Lsc\Lc) \left(\mathrm{dist}(y, \M)\right)^2\\
				\overset{(ii)}{\geq}{}&  \left(\Lsc- \frac{2(\Lc\Mc + \Lsc\Lc)}{\Lsc}\mathrm{dist}(y, \M)\right) \norm{c(y)} \overset{(iii)}{\geq}  \frac{\Lsc}{2} \norm{c(y)},
			\end{aligned}
		\end{equation}
		where $(i)$ holds from the definition of $z$ and $(ii)$ directly follows from Lemma \ref{Le_bound_cx}. Moreover, $(iii)$ holds from the definitions for $\Omegax{x}$ and $\varepsilon_x$, which illustrate that 
		\begin{equation*}
			\begin{aligned}
				&\mathrm{dist}(y, \M) \leq \norm{y-x} \leq  \varepsilon_x \\
				=& \min \left\{ \frac{\rho_{x}}{2}, \frac{\Lsc}{32\Lc(\Ma+1)}, \frac{\Lsc ^2}{8(\Lac + \Lc )\Mc } \right\}\\
				\leq&  \left\{\frac{\Lsc^2}{32\Lsc\Lc}, \frac{\Lsc^2}{8\Lc\Mc} \right\} \leq \frac{\Lsc^2}{4\Lc\Mc + 16\Lsc\Lc}\\
				\leq& \frac{\Lsc^2}{4(\Lc\Mc + \Lsc\Lc)}. 
			\end{aligned}
		\end{equation*}
		This completes the proof.
	\end{proof}

	In the rest of this subsection, we present the theoretical property of $\A^{\infty}$ and $\Ja$ in Lemma \ref{Le_dist_Ainfty} - Lemma \ref{Le_Ja_idempotent}, whose proofs directly follow the routines in  \cite[Section 3]{xiao2022constraint} and thus are omitted for simplicity. 
	\begin{lem}
		\label{Le_dist_Ainfty}
		For any given $x \in \M$ and any $\y \in \BOmegax{x}$, $\A^{\infty}(\y)$ exists and  $\A^{\infty}(\y) \in \Omegax{x} \cap \M$. Moreover,  it holds that
		\begin{equation}
			\norm{\A^{\infty}(\y) - \y} \leq \frac{4(\Ma +1)}{\Lsc } \norm{c(\y)}. 
		\end{equation}
	\end{lem}

	\begin{lem}
		\label{Le_Ja_identical}
		For any given $x \in \M$, the  inclusion 
		$\Ja(x)\tp d  \in \Tx$
		holds for any $d  \in \bb{R}^n$. 
		Moreover, when $d \in \Tx$, it holds that
		$\Ja(x)\tp d = d$. 
	\end{lem}

	\begin{lem}
		\label{Le_Ja_nullspace}
		For any given $x \in \M$, the equality 
		$\Ja(x)d = 0$ holds if and only if
		$d \in \Nx.$
	\end{lem}

	\begin{lem}
		\label{Le_Ja_idempotent}
		For any given $x \in \M$, it holds that
		$\Ja(x)^2 = \Ja(x)$.
	\end{lem}

	\subsection{Selection of the penalty paramete}
	In this subsection, we discuss how to practically choose an appropriate penalty parameter for \ref{Prob_Pen}. We first consider the threshold value for the penalty parameter $\beta$ in \ref{Prob_Pen} defined in the following definition.  
	\begin{defin}
		\label{Cond_beta_esti}
		For any given $x \in \M$, we set
		\small{
			\begin{equation*}
				\betaesti_x := \sup_{y \in \Omegax{x}}\left\{  \max\left\{\frac{2(f(\A^2(y)) - f(\A(y)))}{\norm{c(y)}^2 - \norm{c(\A(y))}^2 }, \frac{ (\A(y) - y )\tp \Ja(y)\nabla f(\A(y)) }{(y - \A(y))\tp \Jc(y) c(y)} , 0  \right\} \right\}. 
		\end{equation*}}
	\end{defin}
	We first show that  $\betaesti_x < +\infty$ for any $x \in \M$ in the following proposition.
	\begin{prop}
		\label{Prop_upper_bound_beta}
		For any $x \in \M$, it holds that 
		\begin{equation*}
			\betaesti_x \leq \max \left\{\frac{64\Lf (\Ma +1)\Lac }{3\Lsc ^3}, 
			\frac{(24\La\Ma + 8\La )\Lf}{\Lsc^2} \right\}.
		\end{equation*}
	\end{prop}
	\begin{proof}
		Firstly, it directly follows from \cite[Proposition 4.2]{xiao2022constraint} that 
		\begin{equation*}
			\sup_{y \in \Omegax{x}} \left\{ \frac{2(f(\A^2(y)) - f(\A(y)))}{\norm{c(y)}^2 - \norm{c(\A(y))}^2 } \right\} \leq  \frac{64\Lf (\Ma +1)\Lac }{3\Lsc ^3}. 
		\end{equation*}
		
		Moreover,  for any $x \in \M$ and any $y \in \Omegax{x}$, let $z \in \mathrm{proj}(y, \M)$, it holds from Lemma \ref{Le_bound_cx} that 
		\begin{equation}
			\label{Eq_Prop_upper_bound_beta_0}
			\begin{aligned}
				&|(y-z)\tp \Ja(y)\Jc(y)c(y)| \leq \norm{y-z}\norm{c(y)} \norm{\Ja(y)\Jc(y)} \\
				\leq{}& \Lac \norm{y-z}\norm{c(y)} \mathrm{dist}(y, \M) = \Lac \norm{y-z}^2\norm{c(y)}. 
			\end{aligned}
		\end{equation}

		Then there exists $\xi = ty + (1-t)z$ for some $t\in [0,1]$ such that  
		\begin{equation*}
			\begin{aligned}
				&(y - \A(y))\tp \Jc(y)c(y) = (y-z)\tp (I_n - \Ja(\xi)) \Jc(y) c(y)\\
				\geq{}& (y-z)\tp \Jc(y) c(y) - \big| (y-z)\tp \Ja(y)\Jc(y)c(y) \big| - \La \Mc \norm{y-z}^2\norm{c(y)}.
			\end{aligned}
		\end{equation*}		
		Now by the mean-value theorem,  for some $\nu = sy + (1-s)z$ with $s\in[0,1]$, we have that
		$c(y)\tp c(y) = (c(y)-c(z))\tp c(y) = (y-z)\tp \Jc(\nu) c(y) = 
		(y-z)\tp \Jc(y)c(y) + (y-z)\tp (\Jc(\nu)-\Jc(y))c(y)$. Thus 
		\begin{equation*}
			\begin{aligned}
				&(y - \A(y))\tp \Jc(y)c(y)  \\
				\geq{} & \norm{c(y)}^2 - \big|(y-z)\tp (\Jc(\nu)-\Jc(y)) c(y)\big|\\
				&- \big| (y-z)\tp \Ja(y)\Jc(y)c(y) \big| - \La \Mc \norm{y-z}^2\norm{c(y)} \\
				\overset{(i)}{\geq}{}& \frac{\Lsc}{2} \norm{y-z}\norm{c(y)}    - \Lc \norm{y-z}^2 \norm{c(y)}\\
				&- L_{x,b}\norm{y-z}^2\norm{c(y)}- \La \Mc \norm{y-z}^2\norm{c(y)} \\
				\overset{(ii)}{\geq}{}& \frac{15\Lsc}{32} \norm{y-z}\norm{c(y)} - ( \La \Mc + \Lac) \norm{y-z}^2\norm{c(y)} \\
				\overset{(iii)}{\geq}{}& \frac{\Lsc}{4}\norm{y-z}\norm{c(y)}. 
			\end{aligned}
		\end{equation*}
		Here $(i)$ uses \eqref{Eq_Prop_upper_bound_beta_0}, $(ii)$ follows from the fact that $\norm{y-z}\leq \varepsilon_x \leq \frac{\Lsc}{32\Lc}$, and $(iii)$ holds from the  fact that $( \La \Mc + \Lac) \norm{y-z} \leq ( \La \Mc + \Lac)\varepsilon_x \leq \frac{1}{5}$.
		Together with Lemma \ref{Le_bound_cx}, we achieve that 
		\begin{equation}
			\label{Eq_Prop_upper_bound_beta_1}
			(y - \A(y))\tp \Jc(y)c(y) \geq \frac{\Lsc}{4\Mc} \norm{c(y)}^2. 
		\end{equation}
		Moreover, from the mean-value theorem, it holds from the Lipschitz continuity of $\Ja(\y)$ and  Lemma \ref{Le_Ja_idempotent} that
		\begin{equation}
			\begin{aligned}
				&|(y-z)\tp (I_n - \Ja(y)) \Ja(y)\nabla f(\A(y))| \\
				\leq{}&  \norm{\Ja(y) - \Ja(y)^2} \Lf \norm{y-z}\\
				\leq{}& \La (2\Ma  + 1) \mathrm{dist}(y, \M) \Lf \norm{y-z}\\
				={}& \La (2\Ma  + 1)\Lf \norm{y-z}^2. 
			\end{aligned}
		\end{equation}
		As a result, we have
		\begin{equation*}
			\begin{aligned}
				&|(y - \A(y))\tp \Ja(y) \nabla f(\A(y)) | \\
				\leq{}& |(y-z)\tp (I_n - \Ja(y)) \Ja(y)\nabla f(\A(y))| \\
				& + |(y-z)\tp (\Ja(\xi) - \Ja(y)) \Ja(y)\nabla f(\A(y))| 
				\\
				\leq{}&  \La (2\Ma  + 1)\Lf \norm{y-z}^2 + \La \Ma \Lf \norm{y-z}^2\\
				\leq{}& \frac{(6\La\Ma +2 \La )\Lf}{\Lsc} \norm{c(y)}\norm{y-z},
			\end{aligned}
		\end{equation*}
		where the first inequality is obtained by applying the mean-value theorem
		to $(y - \A(y))\tp \Ja(y) \nabla f(\A(y)) = \big((y - \A(y)) - (z-\A(z))\big)\tp \Ja(y) \nabla f(\A(y)) = (y-z)\tp (I_n-\Ja(\xi))\Ja(y) \nabla f(\A(y))$  for some $\xi = ty+(1-t)z$ such that $t\in[0,1]$.
		
		Therefore, for any $y \in \Omegax{x}$, we obtain
		\begin{equation*}
			\begin{aligned}
				\left|\frac{ (y - \A(y))\tp \Ja(y)\nabla f(\A(y)) }{(y - \A(y))\tp \Jc(y) c(y)} \right|
				\leq{}& \frac{\frac{(6\La\Ma + 2\La )\Lf}{\Lsc} \norm{c(y)}\norm{y-z}}{\frac{\Lsc}{4}\norm{c(y)}\norm{y-z}} \\
				\leq{}& \frac{(24\La\Ma + 8\La )\Lf}{\Lsc^2}.
			\end{aligned}
		\end{equation*}
		This completes the proof.
	\end{proof}

	\begin{rmk}
		\label{Rmk_esti_beta}
		
		As illustrated in Proposition \ref{Prop_upper_bound_beta},  $\betaesti_x$ is well-defined for any given $x \in \M$. 
		To estimate the value of $\betaesti_x$ from $f$, $\A$, $c$ and their derivatives, we can employ various approaches for estimating $\betaesti_x$, such as Monte Carlo sampling, Markov chain Monte Carlo (MCMC) sampling, etc.  For example, when employing the  Monte Carlo sampling method for choosing the penalty parameter for \ref{Prob_Pen}, we can first choose a reference point $\tilde{x} \in \M$ and randomly sample $N_{\beta}$ points  $\{x^{ref}_1,...,x^{ref}_{N_{\beta}}\} \in \ca{B}(\tilde{x}, \delta_{\beta}):= \{ x \in \bb{R}^n: \norm{x - \tilde{x}} \leq \delta_{\beta} \}$. Then we compute an estimated value  for $\betaesti_x$, denoted as $\tilde{\beta}^*_{\tilde{x}}$, by the following scheme, 
		\begin{equation}
			\label{Eq_esti_beta}
			\begin{aligned}
				\tilde{\beta}^*_{\tilde{x}} = \theta_{\beta} \cdot \max_{1\leq i\leq N_{\beta}}\biggl\{  
				& \max\biggl\{\frac{ 2(f(\A^2(x^{ref}_i)) - f(\A(x^{ref}_i)) )}{\left|\norm{c(x^{ref}_i)}^2 -\norm{c(\A(x^{ref}_i))}^2 \right|  + \varepsilon_{\beta} }, \\
				& \frac{ (\A(x^{ref}_i) - x^{ref}_i )\tp \Ja(x^{ref}_i)\nabla f(\A(x^{ref}_i)) }{ \left|(x^{ref}_i - \A(x^{ref}_i))\tp \Jc(x^{ref}_i) c(x^{ref}_i) \right| + \varepsilon_{\beta} }, 0  \biggr\} \biggr\}. 
			\end{aligned}
		\end{equation}
		Here $\delta_{\beta}>0$, $\varepsilon_{\beta} \geq  0$ and $\theta_{\beta} \geq 1$ are some prefixed hyper-parameters.

		When Assumption \ref{Assumption_4} holds, it is easy to verify that the function 
		\begin{equation}
			\phi_{\beta,x}(y) := \max\left\{\frac{2(f(\A^2(y)) - f(\A(y)))}{\norm{c(y)}^2 - \norm{c(\A(y))}^2 }, \frac{ (\A(y) - y )\tp \Ja(y)\nabla f(\A(y)) }{(y - \A(y))\tp \Jc(y) c(y)} , 0  \right\}
		\end{equation}
		is continuous over $\Omega_x \setminus x$ for any $x \in \M$. Therefore, with $\delta_{\beta}\geq \sup_{y \in \Omega_x} \norm{y-x}$, $\theta_{\beta} \geq 1$, and $\varepsilon_{\beta} = 0$,  our sampling technique guarantees that  $\lim_{N_{\beta}\to +\infty}\tilde{\beta}^*_{x} \geq \betaesti_x$. Moreover, when we assume that $\phi_{\beta,x}$ is locally Lipschitz over $\Omega_x$, the results in \cite{marion2023finite} ensure that for any $\varepsilon > 0$, choosing $N_{\beta} = \ca{O}(1/\varepsilon^{2})$ guarantees that $\tilde{\beta}^* \geq \betaesti_x - \varepsilon$ with probability at least $\frac{3}{4}$.

	\end{rmk}

	\begin{rmk}
		As demonstrated in Definition \ref{Cond_beta_esti},  the penalty parameter $\betaesti_x$ is defined in the neighborhood $\Omegax{x}$ for any given reference point $x \in \M$.  Therefore, the theoretical threshold value $\tilde{\beta}_{\tilde{x}}$ is guaranteed to be effective only in a neighborhood of the reference point $\tilde{x}$. For those cases where $\M$ is a compact manifold, we can choose $\delta_{\beta}$ to be sufficiently large so that $\M \subset \ca{B}(\tilde{x}, \delta_{\beta})$. Therefore, with  $\beta > \tilde{\beta}^*_{\tilde{x}}$, we can conclude that there exists a neighborhood of $\M$, where CDF and OCP have the same first-order and second-order stationary points. Furthermore, in a wide range of popular optimization methods,  the function value sequence of the generated iterates is non-increasing. That is, when we further assume the coercivity of the objective function $f$, there exists $M_{x} > 0$ such that the sequence $\{\xk\}$ is restricted within $\ca{B}(0, {M_x})$. Under such settings, we can choose $\delta_{\beta}$ to enforce the equivalence between CDF and OCP in a neighborhood of $\M \cap \ca{B}(0, {M_x})$. 
	\end{rmk}

	\subsection{Relationships between \ref{Prob_Ori} and \ref{Prob_Pen}}
	This subsection investigates the relationships between \ref{Prob_Ori} and \ref{Prob_Pen} on their stationary points and local minimizers. We start by presenting the explicit expression for the gradient and Hessian of \ref{Prob_Pen}. 
	
	\begin{prop}
		\label{Prop_gradient_h}
		The gradient of $h$ in \ref{Prob_Pen} can be expressed as 
		\begin{equation}
			\nabla h(x) = \Ja(x)\nabla f(\A(x)) + \beta \Jc(x) c(x). 
		\end{equation}
		Furthermore, under Assumption \ref{Assumption_4}, the Hessian of $g(x) := f(\A(x))$ can be expressed as
		\begin{equation}
			\label{Eq_Prop_gradient_h_0}
			\nabla^2 g(x) = \Ja(x)\nabla^2 f(\A(x)) \Ja(x)\tp + \DJa(x)[\nabla f(\A(x))].
		\end{equation}
		and the Hessian of $h(x)$ is
		\begin{equation}
			\nabla^2 h(x) = \nabla^2 g(x) + \beta \left(\Jc(x)\Jc(x)\tp + \DJc(x)[c(x)]\right). 
		\end{equation}
	\end{prop}
	The proof for Proposition \ref{Prop_gradient_h} is straightforward from the expression of $h$ in \ref{Prob_Pen}, hence we omit it for simplicity. 
	
	\begin{prop}	
		\label{Prop_Firstorder_Equivalence_feasible}
		For any $x \in \M$, $x$ is a first-order stationary point of \ref{Prob_Ori} if and only if $x$ is a first-order stationary point of \ref{Prob_Pen}.
	\end{prop}
	\begin{proof}
		For any first-order stationary point, $x \in \M$ of \ref{Prob_Ori}, it follows from the equality
		\eqref{add:2} that
		\begin{equation*}
			\nabla h(x) = \Ja(x) \nabla f(x) + \beta \Jc(x) c(x) = \Ja(x)\Jc(x) \lambda(x) = 0,
		\end{equation*}
		which implies that $x$ is a first-order stationary point of \ref{Prob_Pen}. 
		
		On the other hand, for any $x \in\M$ that is a first-order stationary point of \ref{Prob_Pen}, it holds that
		\begin{equation*}
			\Ja(x) \nabla f(x) = 0,
		\end{equation*}
		which results in the inclusion $\nabla f(x) \in \Nx =  \mathrm{range}(\Jc(x))$ from Lemma \ref{Le_Ja_nullspace}. By Definition \ref{Defin_FOSP}, we conclude that $x$ is a first-order stationary point of \ref{Prob_Ori}. 
	\end{proof}

	\begin{theo}
		
		\label{The_firstorder_equivalence}
		For any given $x \in \M$, suppose $\beta > \betaesti_x$, then any first-order stationary point of \ref{Prob_Pen} in $\Omegax{x}$ is a first-order stationary point of \ref{Prob_Ori}. 
	\end{theo}
	\begin{proof}
		For any $y \in \Omegax{x}$ that is a first-order stationary point of \ref{Prob_Pen}, it holds from Proposition \ref{Prop_gradient_h} that 
		\begin{equation}
			0 = (y - \A(y))\tp \nabla h(y) =  (y - \A(y))\tp \Ja(y) \nabla f(\A(y)) + \beta (y - \A(y))\tp \Jc(y) c(y). 
		\end{equation}

		Moreover, from Definition \ref{Cond_beta_esti} , it holds that
		\begin{equation}
			(y - \A(y))\tp \Ja(y) \nabla f(\A(y)) \geq -\betaesti_x (y - \A(y))\tp \Jc(y) c(y).
		\end{equation}
		Note that we have used the fact that 
		$(y - \A(y))\tp \Jc(y)c(y) \geq 0$ from the proof of Proposition \ref{Prop_upper_bound_beta}.
		Therefore, combined with \eqref{Eq_Prop_upper_bound_beta_1}, we obtain 
		\begin{equation}
			\begin{aligned}
				0 ={}&   (y - \A(y))\tp \Ja(y) \nabla f(\A(y)) + \beta (y - \A(y))\tp \Jc(y) c(y)\\
				\geq{}& -\betaesti_x (y - \A(y))\tp \Jc(y) c(y) +  \beta (y - \A(y))\tp \Jc(y) c(y)\\
				\geq {}&   \frac{\Lsc(\beta - \betaesti_x)}{4\Mc} \norm{c(y)}^2,
			\end{aligned}
		\end{equation}
		which illustrates that $\norm{c(y)} = 0$ and thus $y \in \M$. Together with Proposition \ref{Prop_Firstorder_Equivalence_feasible}, we conclude that $y$ is a first-order stationary point of \ref{Prob_Ori}, and the proof is completed. 
		
	\end{proof}

	\begin{theo}
		
		\label{The_secondorder_equivalence}
		Suppose Assumption \ref{Assumption_4} holds. Then for any given $x \in \M$ with $\beta \geq \betaesti_x$ in \ref{Prob_Pen}, \ref{Prob_Ori} and \ref{Prob_Pen} have the same second-order stationary point over $\Omega_x$.

	\end{theo}
	\begin{proof}
		For the first part of the proof, we aim to show that under Assumption \ref{Assumption_4}, any second-order stationary point of \ref{Prob_Pen}  is a second-order stationary point of
		\ref{Prob_Ori}. For any $y \in \Omega_x$ that is a second-order stationary point of \ref{Prob_Pen}, $y$ is also a first-order stationary point of \ref{Prob_Pen}.
		Then based on Theorem \ref{The_firstorder_equivalence}, $y$ is a first-order stationary point
		of \ref{Prob_Ori}, and 
		it holds that $y \in \M$. Then from Lemma \ref{Le_Ja_identical}, $\Ja(y) \tp d = d$ holds for any $d \in \ca{T}_y$,  which implies the equality
		\begin{equation*}
			d\tp \nabla^2 h(y) d  = d\tp \left( \nabla^2 f(y) - \sum_{i=1}^p \lambda_i(y) \nabla^2 c_i(y) \right) d + d\tp \left(\DJa(y)[\grad f(y)]\right) d.
		\end{equation*}
		Since $y\in \M$ is a first-order stationary point of \ref{Prob_Ori}, we have
		that $\grad f(y) =0.$ Hence
		together with Proposition \ref{Prop_FOSP_Rie}, we obtain the follow inequality,
		\begin{equation*}
			\begin{aligned}
				&\lambda_{\min}(\hess f(y)) \geq \min_{d \in  \ca{T}_y, \norm{d} = 1} ~ d\tp \nabla^2 h(y) d  \geq  \lambda_{\min}(\nabla^2 g(y))\\
				={}& \lambda_{\min}(\nabla^2 h(y))  \geq 0. 
			\end{aligned}
		\end{equation*} 
		Therefore, we can conclude that  $y$ is a second-order stationary point of \ref{Prob_Ori}. This completes the first part of the proof.

		We prove the second part of this theorem by contradiction. Suppose $x$ is a second-order stationary point of \ref{Prob_Ori} but is not a second-order stationary point for \ref{Prob_Pen}. Then there exists a constant $\tau > 0$ and a sequence $\{y_i\} \subset \BOmegax{x}$ such that $\{y_i\} \to x$ and $h(y_i) < h(x) - \tau \norm{y_i - x}^2$ holds for any $i \geq 0$.

		From the inequality 
		\begin{equation*}
			\beta \geq 2\cdot \sup_{y \in \Omegax{x}}\left\{  \max\left\{\frac{2(f(\A^2(y)) - f(\A(y)))}{\norm{c(y)}^2 - \norm{c(\A(y))}^2 }, 0  \right\} \right\},
		\end{equation*}
		it is easy to verify that the following inequality holds for any $y \in \Omegax{x}$, 
		\begin{equation}
			\label{Eq_Prop_esti_beta_0}
			h(\A(y)) = f(\A^2(y)) + \frac{\beta }{2} \norm{c(\A(y))}^2 \leq f(\A(y)) + \frac{\beta}{2}\norm{c(y)}^2  =h(y). 
		\end{equation}
		From \eqref{Eq_Prop_esti_beta_0}, for  any $i\geq 0$, it holds that 
		\begin{equation}
			\label{Eq_Prop_esti_beta_1}
			f(\A^{\infty}(y_i)) = h(\A^{\infty}(y_i)) \leq h(\A(y_i)) \leq h(y_i) < h(x) - \tau\norm{y_i - x}^2. 
		\end{equation}
		
		Furthermore,   notice that 
		\begin{equation}
			\norm{\A^{\infty}(y_i) - y_i} \leq \frac{4(\Ma +1)}{\Lsc } \norm{c(y_i)}, 
		\end{equation}
		then  it holds that 
		\begin{equation}
			\label{Eq_Prop_esti_beta_2}
			\begin{aligned}
				&\norm{\A^{\infty}(y_i) - x} \leq \frac{4(\Ma +1)}{\Lsc } \norm{c(y_i)} + \norm{y_i - x}\\
				\leq{}& \frac{4(\Ma +1)\Lc}{\Lsc } \mathrm{dist}(y_i, \M) + \norm{y_i - x}\\
				\leq{}& \left(\frac{4(\Ma +1)\Lc}{\Lsc } + 1 \right) \norm{y_i - x}. 
			\end{aligned}
		\end{equation}
		
		Therefore,  from  \eqref{Eq_Prop_esti_beta_1}, it holds that $\{\A^{\infty}(y_i)\} \to x$ and  for any $i \geq 0$,
		\begin{equation*}
			f(\A^{\infty}(y_i)) \leq f(x) - \frac{\tau}{\left(\frac{4(\Ma +1)\Lc}{\Lsc } + 1 \right)^2} \norm{\A^{\infty}(y_i) - x}^2,
		\end{equation*}
		which contradicts the fact that $x$ is a second-order stationary point of \ref{Prob_Ori} as illustrated in \cite[Proposition 5.5.5]{Absil2009optimization}. Therefore, from the contradiction, we can conclude that $x$ is a second-order stationary point for \ref{Prob_Pen}. This completes the entire proof. 
		
		\
	\end{proof}

	\section{Software Description}
	In this section, we give a brief overview of the structure of \CDOpt. The idea behind the design for \CDOpt is to guarantee the ease-of-use features for various types of users, including the users from optimization communities for solving standard Riemannian optimization problems over commonly used manifolds, the users from machine learning communities who wish to train neural networks with manifold constraints, and developers who aim to develop new manifolds and introduce more advanced features for \CDOpt. As illustrated in Figure \ref{Fig:sketch_structure}, we organize the \CDOpt package into the following three parts: 
	\begin{itemize}
		\item The manifold classes;
		\item The problem classes;
		\item The neural layers.
	\end{itemize}
	These three parts are mutually separated, hence it is easy to add new manifolds, and new neural layers in \CDOpt with maximum code reuse for developers.  
	
	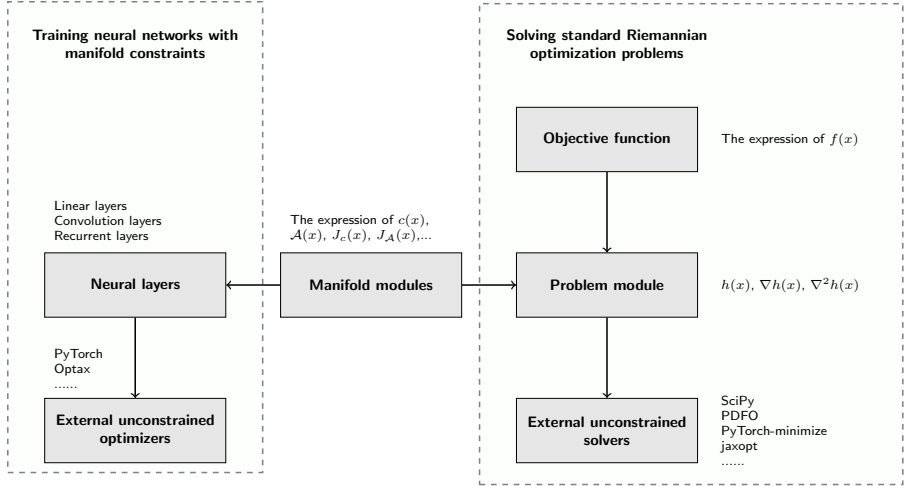
\begin{figure}[htbp]
		\centering
		\resizebox{\columnwidth}{!}{%
			\begin{tikzpicture}
				[node distance = 1cm, auto,font=\footnotesize,
				% STYLES
				every node/.style={node distance=3cm},
				% The comment style is used to describe the characteristics of each force
				comment/.style={rectangle, inner sep= 5pt, text width=3cm, node distance=0.25cm, font=\scriptsize\sffamily},
				refpoint/.style={rectangle, inner sep= 0pt, text width=0cm, node distance=0cm, font=\scriptsize\sffamily},
				outer/.style={draw=gray,dashed,fill=green!1,thick,inner sep=5pt},
				textbox/.style={rectangle,inner sep=5pt, text width=4cm, text badly centered, minimum height=1.2cm, font=\bfseries\footnotesize\sffamily},
				% The force style is used to draw the forces' name
				force/.style={rectangle, draw, fill=black!10, inner sep=5pt, text width=3cm, text badly centered, minimum height=1.2cm, font=\bfseries\footnotesize\sffamily}] 
				
				% Draw forces
				\node [force] (rivalry) {Problem module};
				\node [force, above = 1.5cm of rivalry] (substitutes) {Objective function};
				\node [force, left=1cm of rivalry] (suppliers) {Manifold modules};
				
				%			\node [force, right=1cm of rivalry] (users) {Bargaining power of users};
				\node [force, below = 1.5cm of rivalry] (entrants) {External unconstrained solvers};
				
				\node [force, left=1.0cm of suppliers] (state) {Neural layers};
				\node [force, below = 1.5cm of state] (optimizers) {External unconstrained optimizers};
				\node (text) [textbox, anchor=north] at ([yshift=4.5cm]state.north) {Training neural networks with manifold constraints};
				
				\node (textnew) [textbox, anchor=north] at ([yshift=1.75cm]substitutes.north) {Solving standard Riemannian optimization problems};

				%			\node [below=0.25  of rivalry] (ref-point)
				%			\node [point, ]
				
				%%%%%%%%%%%%%%%
				% Change data from here
				
				\node [comment, above=0.0cm of suppliers] {The expression of $c(x)$,  $\A(x)$, $\Jc(x)$, $\Ja(x)$,...};
				
				\node [refpoint, below=0.75cm of suppliers](comment-manifold) {};
				
				\node [refpoint, below=0.75cm of state](point1) {};
				
				% RIVALRY
				%			\node [comment, right=4.25cm of comment-manifold] (comment-rivalry) {$h(x)$, $\nabla h(x)$, $\nabla^2 h(x)$};
				\node [comment, right=0.25cm of rivalry] (comment-rivalry) {$h(x)$, $\nabla h(x)$, $\nabla^2 h(x)$};

				% SUPPLIERS

				% SUBSTITUTES
				\node [comment, right=0.25 of substitutes] { The expression of $f(x)$ };
				
				%			% USERS
				%			\node [comment, below=0.25 of users] {(+) Increasing the user information\\
					%				(+) Reducing the switching costs};
				
				% NEW ENTRANTS
				\node [comment, right=0.25 of entrants] (comment-optimizers) {\Pkg{SciPy}\\ \Pkg{PDFO}\\ \Pkg{PyTorch-minimize}\\ \Pkg{jaxopt}\\......};
				
				% PUBLIC POLICIES
				\node [comment, text width=3cm, above=0.0 of state] {Linear layers\\ Convolution layers \\Recurrent layers};
				
				\node [comment, text width=3cm, above=0.0 of optimizers] {\Pkg{PyTorch}\\ \Pkg{Optax}\\ ......};

				\begin{pgfonlayer}{background}
					\node[outer,fit=(state) (optimizers) (text)] (A) {};
				\end{pgfonlayer}

				\begin{pgfonlayer}{background}
					\node[outer,fit=(comment-optimizers)  (textnew)] (A2) {};
				\end{pgfonlayer}
				
				%			\node [comment, right=0.25 of optimizers] {The expression of $f(x)$}
				%%%%%%%%%%%%%%%%
				
				% Draw the links between forces
				\path[->,thick] 
				(substitutes) edge (rivalry)
				(suppliers) edge (rivalry)
				%			(users) edge (rivalry)
				(rivalry) edge (entrants)
				(suppliers) edge (state)
				(state) edge (optimizers);

				%			\path[->,thick, dashed] 

			\end{tikzpicture} 
		}
		\caption{The sketch of the structure of \CDOpt. }
		\label{Fig:sketch_structure}
	\end{figure}
	
	\subsection{The manifold classes}
	The module ``\codeobj{manifold}'' in \CDOpt describes the manifold constraints by the following essential components for describing the corresponding constraint dissolving mapping,
	\begin{itemize}
		\item The expression of constraints $c(x)$, $\Jc(x)$ and the Hessian-vector product of $\norm{c(x)}^2$;
		\item The expression of $\A(x)$, $\Ja(x)$ and the second-order differential of $\A(x)$;
		\item The rules to compute  $h(x)$, $\nabla h(x)$, and $\nabla^2 f(x)$ from $f$, $\A$, $c$ and their derivatives.
	\end{itemize}	
	All the manifolds in \CDOpt are available as derived classes from the base class named \codeobj{basic\_manifold}, which is available in the \codeobj{manifold} module of \CDOpt. Note that all the essential components in these manifold classes can be expressed by $c(x)$, $\A(x)$ and their differentials. Therefore,  if some of these materials are not provided in the derived classes of the \codeobj{basic\_manifold}, all the essential materials can be automatically calculated based on the selected automatic differentiation (AD) packages from the supported numerical backends.
	In particular, if only the expression of $c(x)$ is provided in the derived class of \codeobj{basic\_manifold}, the constraint dissolving mapping is automatically chosen as 
	\begin{equation}
		\label{Eq_general_cd_mapping}
		\A_c(x) = x - \Jc(x) (\Jc(x)\tp \Jc(x) + \sigma \norm{c(x)}^2  I_p)^{\dagger} c(x).
	\end{equation}
	Here $\Jc(x) \in \bb{R}^{n\times p}$ denotes the transpose of the Jacobian of $c$ at $x$, which is also computed by the selected  AD packages if it is not provided. Moreover, it is easy to verify that $\A_c$ is locally Lipschitz over $\bb{R}^n$ when $\Jc$ is locally Lipschitz smooth over $\bb{R}^n$.  As a result, users only need to provide the expression of the manifold constraints $c(x)$ and its corresponding numerical backends for defining new manifold constraints though \codeobj{basic\_manifold} from the \codeobj{manifold} module in \CDOpt.

	\begin{lstlisting}[caption= Define customized Riemannian manifold in \CDOpt, label = Lst_define_manifold,captionpos=b]
		class customized_manifold(basic_manifold_np):
		def __init__(self, var_shape, T):
		m, s = *var_shape
		self.T = T
		self.Is = np.eye(s)
		super().__init__('customized_manifold',(m,s), (s,s))
		
		def C(self, X):
		return (X.T @ X)@ self.T  - self.Is
	\end{lstlisting}
	
	\CDOpt supports various numerical backends, where the type of the variables, APIs for build-in functions, and AD packages differ among different numerical backends. As a result, \CDOpt provides some derived classes from \codeobj{basic\_manifold}:  \codeobj{basic\_manifold\_np}, \codeobj{basic\_manifold\_torch} and \codeobj{basic\_manifold\_jax}. In these derived classes, only the expression of $c(x)$ is required in defining a manifold, and their numerical backends are fixed as \Pkg{NumPy} + \Pkg{Autograd}, \Pkg{PyTorch} and \Pkg{JAX}, respectively. Then various specific manifolds, with different numerical backends, are developed by inheriting \codeobj{basic\_manifold\_np}, \codeobj{basic\_manifold\_torch} or \codeobj{basic\_manifold\_jax}. Listing \ref{Lst_define_manifold} illustrates how to define the Riemannian manifold $\{ X \in \bb{R}^{m\times s}: X\tp X T = I_s \}$ in \CDOpt only by the expression of the constraints $X\tp X T - I_s = 0$.

	Furthermore, \CDOpt provides various predefined manifolds, where the essential components are prefixed manually to achieve the best computational efficiency. Table \ref{Table_cdopt_manifold} illustrates these predefined manifolds in \CDOpt. Moreover, the formulation of their corresponding constraint dissolving mappings in \CDOpt is listed in Table \ref{Table_implementation}. 
	
	\begin{rmk}
		In \CDOpt, each point on the Grassmann manifold is an $s$-dimensional subspace embedded in $\bb{R}^m$ and represented by an orthogonal $m \times s$ matrix. Consequently, optimization problems over the Grassmann manifold in \CDOpt are treated as a specific class of optimization problems over the Stiefel manifold. Since the constraint dissolving mapping $\A$ only depends on the constraints, the constraint dissolving mapping $\mathcal{A}$ for the Grassmann manifold in \CDOpt shares the same expression as that for the Stiefel manifold \cite{xiao2021solving,xiao2022constraint}. 
		
		Additionally, compared to optimization problems over the Stiefel manifold in \CDOpt, those over the Grassmann manifold require the objective functions to be orthogonally invariant. That is, for an objective function $f: \mathbb{R}^{m \times s} \to \mathbb{R}$, the equality $f(X) = f(XQ)$ holds for any $X \in \mathbb{R}^{m \times s}$ and any orthogonal matrix $Q \in \mathbb{R}^{s \times s}$. In contrast, the optimization problems over the Stiefel manifold admit more general functions. This leads to the different solution properties between these two classes of problems in \CDOpt. 
		
	\end{rmk}

	\begin{table}[htbp]
		\tiny
		\centering
		\caption{Predefined manifold in \CDOpt. Here ${\bf 0}_{m\times m}$ denotes the $m\times m$  zero matrix, and $X^H$ denotes the conjugate transpose of a complex matrix $X$.}
		\label{Table_cdopt_manifold}
		\scalebox{0.8}{
			\begin{tabular}{|lll|}
				\hline
				Name & $\M$ & Predefined manifolds\\
				\hline
				\multirow{3}{*}{Euclidean space} & \multirow{3}{*}{$\bb{R}^n$} & \codeobj{eudliean\_np}\\
				& & \codeobj{eudliean\_torch}\\
				& & \codeobj{eudliean\_jax}\\ \hline
				\multirow{3}{*}{Sphere} & \multirow{3}{*}{$\{ x \in \bb{R}^n: \norm{x} = 1 \}$} & \codeobj{sphere\_np}\\
				& & \codeobj{sphere\_torch}\\
				& & \codeobj{sphere\_jax}\\ \hline
				\multirow{3}{*}{Oblique manifold} & \multirow{3}{*}{$\{ X \in \bb{R}^{m\times s}: \mathrm{Diag}(XX\tp) = I_m \}$} & \codeobj{oblique\_np}\\
				& & \codeobj{oblique\_torch}\\
				& & \codeobj{oblique\_jax}\\ \hline
				\multirow{3}{*}{Stiefel manifold} & \multirow{3}{*}{ $\{ X \in \bb{R}^{m\times s}: X^TX = I_s \}$} & \codeobj{stiefel\_np}\\
				& & \codeobj{stiefel\_torch}\\
				& & \codeobj{stiefel\_jax}\\ \hline
				\multirow{3}{*}{Grassmann manifold} & \multirow{3}{*}{ $\{ \mathrm{range}(X): X \in \bb{R}^{m\times s}, X^TX = I_s \}$}  & \codeobj{grassmann\_np}\\
				& & \codeobj{grassmann\_torch}\\
				& & \codeobj{grassmann\_jax}\\ \hline
				\multirow{3}{*}{Generalized Stiefel manifold} & \multirow{3}{*}{ $\{ X \in \bb{R}^{m\times s}: X^T B X = I_s \}$, $B \succeq 0$ } & \codeobj{generalized\_stiefel\_np}\\
				& & \codeobj{generalized\_stiefel\_torch}\\
				& & \codeobj{generalized\_stiefel\_jax}\\ \hline
				\multirow{3}{*}{Hyperbolic manifold} & \multirow{3}{*}{ $\{ X \in \bb{R}^{m\times s}: X^T B X = I_s\}$, $B$ is indefinite } & \codeobj{hyperbolic\_np}\\
				& & \codeobj{hyperbolic\_torch}\\
				& & \codeobj{hyperbolic\_jax}\\ \hline
				\multirow{3}{*}{Symplectic Stiefel manifold} & \multirow{3}{*}{ $\left\{ X \in \bb{R}^{2m\times 2s}: X \tp Q_m X = Q_s \right\}$, $Q_m := \left[ \begin{smallmatrix}
						{\bf 0}_{m\times m} & I_m\\
						-I_m & {\bf 0}_{m\times m}
					\end{smallmatrix}\right]$ } & \codeobj{symp\_stiefel\_np}\\
				& & \codeobj{symp\_stiefel\_torch}\\
				& & \codeobj{symp\_stiefel\_jax}\\ \hline
				\multirow{3}{*}{Stiefel manifold with range constraint} & \multirow{3}{*}{ $\{ X \in \bb{R}^{m\times s}: X\tp X = I_s, ~ XX\tp e = e\}$, $\norm{e} = 1$.}  & \codeobj{stiefel\_range\_constraints\_np}\\
				& & \codeobj{stiefel\_range\_constraints\_torch}\\
				& & \codeobj{stiefel\_range\_constraints\_jax}\\ \hline
				\multirow{3}{*}{Complex sphere} & \multirow{3}{*}{ $\{ x \in \bb{C}^{n}: \norm{x} = 1 \}$} & \codeobj{complex\_sphere\_np}\\
				& & \codeobj{complex\_sphere\_torch}\\
				& & \codeobj{complex\_sphere\_jax}\\ \hline
				\multirow{3}{*}{Complex oblique manifold} & \multirow{3}{*}{ $\{ X \in \bb{C}^{m\times s}: \mathrm{Diag}(XX^H) = I_m\}$ } & \codeobj{complex\_oblique\_np}\\
				& & \codeobj{complex\_oblique\_torch}\\
				& & \codeobj{complex\_oblique\_jax}\\ \hline
				\multirow{3}{*}{Complex Stiefel manifold} & \multirow{3}{*}{ $\{ X \in \bb{C}^{m\times s}: X^H X = I_s\}$ } & \codeobj{complex\_stiefel\_np}\\
				& & \codeobj{complex\_stiefel\_torch}\\
				& & \codeobj{complex\_stiefel\_jax}\\ \hline
				Product manifold & $\M_1\times M_2\times \cdots$ &  \codeobj{product\_manifold} \\\hline
			\end{tabular}
		}
	\end{table}

	\begin{table}[htbp]
		
		\tiny
		\centering
		
		\caption{Default implementations of the constraint dissolving mappings $\A$ for the predefined Riemannian manifolds in \CDOpt. Here ${\bf 0}_{m\times m}$ denotes the $m\times m$ zero matrix, and $X^H$ denotes the conjugate transpose of a complex matrix $X$.} 
		
		\label{Table_implementation}
		\scalebox{0.7}{
			\begin{tabular}{|lll|}
				\hline
				Name of the manifold & Expression of $\M$ & Default choice of $\A$ in \CDOpt \\ \hline
				Euclidean space & $\{x: x \in \bb{R}^n\}$ & $x\mapsto x$ \\ \hline
				Sphere & $\left\{ x \in \bb{R}^{n}:  x\tp x = 1 \right\}$ &  $x \mapsto 2x/(1 + \norm{x}_2^2)$ \cite{xiao2022constraint} \\ \hline
				Oblique manifold & $\left\{ X \in \bb{R}^{m\times s}: \mathrm{Diag} (X X\tp) = I_m \right\}$ &  $X \mapsto 2\left( I_m + \mathrm{Diag}(X X\tp) \right)^{-1} X $ \cite{xiao2022constraint} \\ \hline
				Stiefel manifold &$\left\{ X \in \bb{R}^{m\times s}: X \tp X = I_s \right\}$& $X \mapsto X\left( \frac{3}{2}I_s - \frac{1}{2} X\tp X \right)$ \cite{xiao2021solving}\\ \hline
				Grassmann manifold &$\left\{ \mathrm{range}(X): X \in \bb{R}^{m\times s}, X \tp X = I_s \right\}$& $X \mapsto X\left( \frac{3}{2}I_s - \frac{1}{2} X\tp X \right)$ \\ \hline
				\multirow{2}{*}{Generalized Stiefel manifold } &  $\left\{ X \in \bb{R}^{m\times s}: X \tp B X = I_s \right\}$ for some  & \multirow{2}{*}{$X \mapsto X\left( \frac{3}{2}I_s - \frac{1}{2} X\tp B X \right)$  \cite{xiao2022constraint}  } \\
				&  positive definite $B$ & \\ \hline
				\multirow{2}{*}{Hyperbolic manifold \cite{bai2014minimization} } &  $\left\{ X \in \bb{R}^{m\times s}: X \tp B X = I_s \right\}$ for some $B$ & \multirow{2}{*}{$X \mapsto X\left( \frac{3}{2}I_s - \frac{1}{2} X\tp B X \right)$ \cite{xiao2022constraint}  } \\
				& that satisfies $\lambda_{\min}(B) < 0 < \lambda_{\max}(B)$ & \\ \hline
				\multirow{2}{*}{Symplectic Stiefel manifold \cite{son2021symplectic}} &  $\left\{ X \in \bb{R}^{2m\times 2s}: X \tp Q_m X = Q_s \right\}$ & \multirow{2}{*}{$X \mapsto X \left(\frac{3}{2}I_{2s} + \frac{1}{2}Q_sX\tp Q_mX \right)$ \cite{xiao2022constraint} 
				}\\  
				&$Q_m := \left[ \begin{smallmatrix}
					{\bf 0}_{m\times m} & I_m\\
					-I_m & {\bf 0}_{m\times m}
				\end{smallmatrix}\right]$  & \\ \hline
				\multirow{2}{*}{Stiefel manifold with range constraint \cite{huang2022riemannian}} &  $\{ X \in \bb{R}^{m\times s}: X\tp X = I_s, ~ XX\tp e = e\}$ & $X \mapsto X \left(\frac{3}{2}I_{s}  - \frac{1}{2} X\tp X + (X\tp X - 2I_s) X\tp ee\tp X  \right)$\\  
				& for some $e \in \bb{R}^m$ such that $\norm{e} = 1$. &   $\qquad \qquad + ee\tp X$ \\ \hline
				Complex sphere & $\left\{ x \in \bb{C}^{n}: \norm{x} = 1 \right\}$  & $x \mapsto 2x/(1 + \norm{x}_2^2)$\\ \hline
				Complex oblique manifold & $\left\{ X \in \bb{C}^{m\times s}: \mathrm{Diag}(XX^H) = I_m \right\}$  & $X \mapsto 2\left( I_m + \mathrm{Diag}(X X\tp) \right)^{-1} X $\\ \hline
				Complex Stiefel manifold & $\left\{ X \in \bb{C}^{m\times s}: X^H X = I_s \right\}$  & $X \mapsto X\left( \frac{3}{2}I_s - \frac{1}{2} X^H X \right)$ \cite{xiao2022constraint} \\ \hline
				%			Generalized cases & $\{ x \in \bb{R}^n: c(x) = 0 \}$ & $x \mapsto x -  \Jc(x) \left( \Jc(x)\tp \Jc(x)\right)^{-1}  c(x)$\\ \hline
			\end{tabular}
		}
		
	\end{table}

	\subsection{Problem class} 
	The \codeobj{problem} class in \CDOpt describes the Riemannian optimization problem \ref{Prob_Ori} by the manifold class for the constraints $c(x)$, the objective function $f(x)$ and its derivatives.   As shown in Listing \ref{Lst_define_problem},  \CDOpt allows users to describe the optimization problem only from the expression of the objective function and the instantiated manifold.
	\begin{lstlisting}[caption= Describe Riemannian optimization problem in \CDOpt, label = Lst_define_problem,captionpos=b]
		manifold = stiefel_torch( (100,10) )
		prob = cdopt.core.problem(manifold, obj_fun, beta = 'auto')		
	\end{lstlisting}
	
	When instantiating the \codeobj{problem} class, the argument \codeobj{obj\_fun} should be a callable object, and the argument \codeobj{manifold} should be the instantiation of the \codeobj{manifold} class in \CDOpt. Other materials for optimization, including the gradient and Hessian of the objective function $f$, are optional in the instantiation of \codeobj{problem} class. If these materials are not provided, the \codeobj{problem} class automatically computes them by the chosen numerical backends behind the interface. 
	
	Furthermore, if the penalty parameter $\beta$ in \ref{Prob_Pen} is not provided in the instantiation of the \codeobj{problem} class, \CDOpt can automatically estimate an appropriate value of $\beta$. When the argument \codeobj{beta} is set as \codeobj{'auto'}, \CDOpt randomly chooses the penalty parameter by the scheme presented in  Remark \ref{Rmk_esti_beta}, where the reference point $\tilde{x} \in \M$ from the \codeobj{Init\_point} method of the \codeobj{manifold} class, and the default values for the hyperparameters are chosen as  $N_{\beta} = 100$, $\delta_{\beta} = 0.01$, $\varepsilon_{\beta} = 10^{-14}$ and $\theta_{\beta} = 2$.  Users can set their customized hyperparameters via the argument \codeobj{autobeta\_args}. 
	
	After instantiation, the \codeobj{problem} class provides the constraint dissolving function \ref{Prob_Pen} and its derivatives as callable objects. In the attributes of the \codeobj{problem} class, \codeobj{cdf\_obj}, \codeobj{cdf\_grad} and \codeobj{cdf\_hvp} are callable objects that return the function value, gradient and Hessian-vector product of the constraint dissolving function, respectively. Their input arguments and outputs are of the same type as the selected numerical backend of the problem. Moreover, to adapt the optimizers that are developed based on \Pkg{NumPy} and \Pkg{SciPy}, the \codeobj{problem} class provides the attributes \codeobj{cdf\_obj\_vec\_np}, \codeobj{cdf\_grad\_vec\_np} and \codeobj{cdf\_hvp\_vec\_np}, where their inputs and outputs are all \Pkg{NumPy} 1D array. Based on the provided attributes from the \codeobj{problem} class, users can directly apply the unconstrained optimization solvers from various existing packages, such as \Pkg{SciPy}, \Pkg{PDFO}, \Pkg{PyTorch-minimize} and \Pkg{jaxopt}.

	\subsection{Neural layers}
	\Pkg{PyTorch} and \Pkg{Flax} are powerful frameworks for training neural networks, where neural networks are built by neural layers from \codeobj{torch.nn} or \codeobj{flax.linen} modules. For those users that aim to train neural networks with manifold constraints, \CDOpt provides various predefined neural layers in its  \codeobj{cdopt.nn} and \codeobj{cdopt.linen} modules for \Pkg{PyTorch} and \Pkg{Flax}, respectively. These predefined layers in \CDOpt preserve the same APIs as the layers from \Pkg{PyTorch} and \Pkg{Flax}, hence users can plug these layers into the neural networks with minimal modification to the standard \Pkg{PyTorch} or \Pkg{Flax} codes.

	The neural layers from \codeobj{cdopt.nn} module are the derived classes of the \codeobj{module} class from \codeobj{torch.nn}, hence can be integrated with any neural network built from the neural layers from \codeobj{torch.nn}. Therefore, users can build neural networks by mixing the layers from both \codeobj{cdopt.nn} and \codeobj{torch.nn}. In the instantiation of the neural layers from \codeobj{cdopt.nn}, users can set the manifold constraints by the \codeobj{manifold\_class} argument and choose the penalty parameter $\beta$ by the \codeobj{penalty\_param} argument. When instantiated, the neural layers in the \codeobj{cdopt.nn} module first call its \codeobj{\_\_init\_\_} method to instantiate the manifold according to the \codeobj{manifold\_class} argument and generate the initial weights on the manifold. After instantiation, the neural layers provide the attributes \codeobj{manifold} and \codeobj{quad\_penalty} to access the instantiated manifolds and the value of $\norm{c(x)}^2$, respectively. Moreover, for any neural network that is built from the neural layers from \codeobj{cdopt.nn} and \Pkg{torch.nn}, we can call the function \codeobj{get\_quad\_penalty()} from \codeobj{cdopt.nn} to compute the sum of all the quadratic penalty terms of its neural layers from \codeobj{cdopt.nn}.

	Currently,  \CDOpt implements the following neural layers for \Pkg{PyTorch} in the \codeobj{cdopt.nn} module,
	\begin{itemize}
		\item \codeobj{Linear}, \codeobj{Bilinear}, \codeobj{LazyLinear};
		\item \codeobj{Conv1d}, \codeobj{Conv2d}, \codeobj{Conv3d};
		\item \codeobj{ConvTranspose1d}, \codeobj{ConvTranspose2d}, \codeobj{ConvTranspose3d};
		\item \codeobj{RNN}, \codeobj{LSTM}, \codeobj{GRU}, \codeobj{RNNCell}, \codeobj{LSTMCell}, \codeobj{GRUCell}.
	\end{itemize}
	Listing \ref{Lst1} presents a simple example, where we build a two-layer neural network by the layers from \Pkg{CDOpt.nn} and \Pkg{PyTorch.nn}. 
	
	\begin{lstlisting}[caption= Building a two-layer neural network with manifold constraints in PyTorch by predefined neural layers from \codeobj{cdopt.nn}., label = Lst1,captionpos=b]
		class Net(nn.Module):
		def __init__(self):
		super(Net, self).__init__()
		self.conv = Conv2d_cdopt(1, 6, 5, 
		manifold_class=stiefel_torch, penalty_param = 0.05)
		self.fc = nn.Linear(84, 10)
		
		def forward(self, x):
		x = F.max_pool2d(F.relu(self.conv(x)), (2, 2))
		x = self.fc(x)
		x = F.log_softmax(x, dim=1)
		return x
	\end{lstlisting}

	Furthermore, for those neural layers that are not available from \codeobj{cdopt.nn}, \CDOpt provides a simple way to add manifold constraints to the parameters of these neural layers. Through the function \codeobj{set\_constraint\_dissolving} from \codeobj{cdopt.nn.utils.set\_constraints}, users can set the manifold constraints to the neural layers by just providing the neural layers, the name of target parameters, and the manifold class. Listing \ref{Lst1_b} presents a simple example on how to set the manifold constraints for the weights in a two-layer neural network.
	
	\begin{lstlisting}[caption= Building a two-layer neural network with manifold constraints in PyTorch by the function \codeobj{set\_constraint\_dissolving} from \codeobj{cdopt.nn.utils.set\_constraints}., label = Lst1_b,captionpos=b]
		class Net(nn.Module):
		def __init__(self):
		super(Net, self).__init__()
		self.conv = nn.Conv2d(1, 6, 5)
		self.fc = nn.Linear(84, 10)
		set_constraint_dissolving(self.conv, 'weight',
		manifold_class=stiefel_torch, penalty_param = 0.05)
		
		def forward(self, x):
		x = F.max_pool2d(F.relu(self.conv(x)), (2, 2))
		x = self.fc(x)
		x = F.log_softmax(x, dim=1)
		return x
	\end{lstlisting}

	For \Pkg{Flax}, \CDOpt implements the following neural layers for training neural networks, 
	\begin{itemize}
		\item \codeobj{Dense}, \codeobj{DenseGeneral};
		\item \codeobj{Conv}, \codeobj{ConvTranspose}.
	\end{itemize}
	It is worth mentioning that the neural layers in \Pkg{Flax} are designed in a different way from those in \Pkg{PyTorch.nn}.  In \Pkg{Flax},  the neural layers are separated from their weights and behave as functions that map the inputs and weights to their outputs by the \codeobj{\_\_call\_\_} method. As a result, although the neural layers from \codeobj{linen} module can be called by the same arguments as those from \Pkg{Flax.linen},  the quadratic penalty term $\norm{c(x)}^2$ is returned by the \codeobj{\_\_call\_\_} method.  
	Listing \ref{Lst2}  illustrates how to build a neural network by the neural layers from \Pkg{CDOpt.linen} and \Pkg{Flax.linen}. 
	\begin{lstlisting}[caption= Building a two-layer neural network with manifold constraints in Flax by predefined neural layers from \codeobj{cdopt.linen}, label = Lst2,captionpos=b]
		class model(flax.linen.Module):
		@flax.linen.compact
		def __call__(self, x):
		x, quad_penalty = cdopt.linen.Conv_cdopt(features=32, kernel_size=(3, 3),
		manifold_class = stiefel_jax)(x)
		x = flax.linen.relu(x)
		x = flax.linen.avg_pool(x, window_shape=(2, 2), strides=(2, 2))
		x = flax.linen.Dense(features=10)(x)
		return x, quad_penalty
	\end{lstlisting}
	\subsection{Optimization through \CDOpt}
	In \CDOpt, the Riemannian optimization problem \ref{Prob_Ori} is transformed into the unconstrained minimization of the constraint dissolving function \ref{Prob_Pen}. Therefore, various existing unconstrained optimization solvers can be directly applied to solve \ref{Prob_Ori} through minimizing \ref{Prob_Pen}.  \CDOpt only requires the following steps for solving the optimization problem through the \codeobj{problem} class, 	
	\begin{enumerate}
		\item Define the objective function $f: \bb{R}^n \to \bb{R}$;
		\item Instantiation of the manifold $\M = \{ x \in \bb{R}^n: c(x) = 0 \}$;
		\item Describe the optimization problem through \codeobj{problem} class;
		\item Call an external solver to minimize the constraint dissolving function. 
	\end{enumerate}
	As far as we tested, \CDOpt supports the following solvers from \Pkg{SciPy}, \Pkg{PDFO} \cite{pdfo2022website}, \Pkg{PyTorch-minimize} \cite{torchmin2022github}, and \Pkg{jaxopt} \cite{jaxopt2021},
	\begin{itemize}
		\item \Pkg{SciPy}: \codeobj{CG}, \codeobj{BFGS}, \codeobj{Newton-CG},  \codeobj{L-BFGS-B}, \codeobj{TNC}, \codeobj{trust-ncg}, and \codeobj{trust-krylov};
		\item \Pkg{PyTorch-minimize}: \codeobj{CG}, \codeobj{L-BFGS}, \codeobj{trust-ncg}, and \codeobj{trust-krylov};
		\item \Pkg{jaxopt}: \codeobj{GradientDescent}, \codeobj{LBFGS}, \codeobj{NonlinearCG}, and its interface for the package \codeobj{scipy.optimize.minimize};
		\item \Pkg{PDFO}:  \codeobj{newuoa}, \codeobj{bobyqa}, \codeobj{lincoa}.
	\end{itemize}

	On the other hand,   \CDOpt provides simple approaches to build and train neural networks that invoke manifold constraints. To impose manifold constraints to the weights of a neural layer,  users only need  the following two steps,
	\begin{enumerate}
		\item Replace that neural layer with its corresponding neural layers from \codeobj{cdopt.nn} and \codeobj{cdopt.linen};
		\item Add the quadratic penalty terms to the loss function in the training process.
	\end{enumerate}
	Then various unconstrained optimizers (i.e. solvers for training neural networks) can be directly applied to train the modified neural network. Currently,  \CDOpt is tested to successfully support the following optimizers,
	\begin{itemize}
		\item All the optimizers from \Pkg{PyTorch} and  \Pkg{PyTorch-optimizer};
		\item All the optimizers from \Pkg{Optax}. 
	\end{itemize}

	\section{Comparison with Existing Packages}
	
	In this section, we present a brief comparison between \CDOpt and existing Python Riemannian optimization packages, including \Pkg{PyManopt}, \Pkg{Geoopt}, \Pkg{McTorch} and \Pkg{GeoTorch}. 
	The major difference between \CDOpt and these existing packages is how the \ref{Prob_Ori} is treated. Developed from constraint dissolving approaches, \CDOpt transfers \ref{Prob_Ori} into the unconstrained minimization of \ref{Prob_Pen}, where the construction of \ref{Prob_Pen} avoids the need of geometrical materials from differential geometry.

	\CDOpt only requires $c(x)$, $\A(x)$, and their derivatives to describe the Riemannian manifold, hence it provides great convenience in defining Riemannian manifolds. As demonstrated in \cite{xiao2022constraint},  it is usually easy to design the constraint dissolving mapping $\A$ for $\M$ since the constraint dissolving mappings only need to satisfy Assumption \ref{Assumption_2}. When only $c(x)$ is provided, we can choose the constraint dissolving mapping as in \eqref{Eq_general_cd_mapping} with appropriate $\sigma \geq 0$. Due to the great convenience in describing the manifolds,  \CDOpt provides a wider range of pre-defined manifold classes than existing Riemannian optimization packages. Moreover, \CDOpt enables users to define new manifolds only from the expression of the constraints, hence defining a new manifold in \CDOpt is significantly easier than the existing Riemannian optimization packages. A detailed comparison of the supported manifolds and essential materials for new manifolds is presented in Table \ref{Table:manifold}.

	As \CDOpt transforms \ref{Prob_Ori} into the unconstrained minimization of the constraint dissolving function, \CDOpt can directly and efficiently utilize existing unconstrained optimization solvers without modifying them to adapt the geometrical material. Currently, \CDOpt is compatible with various existing unconstrained optimization solvers. For the optimization problem defined by the \codeobj{problem} class, \CDOpt can directly call all the solvers provided by \Pkg{SciPy}, \Pkg{PyTorch-minimize}, \Pkg{PDFO}, \Pkg{jaxopt} and various other unconstrained optimization packages. For training neural networks, \CDOpt is compatible to the optimizers from  \Pkg{PyTorch}, \Pkg{PyTorch-optimizer}, \Pkg{Optax}, and various open-sourced optimizers.   However, most existing Riemannian optimization packages have to develop specialized solvers to utilize the geometrical materials. Therefore, their supported solvers are relatively limited. Detailed comparisons are presented in Table \ref{Table:solvers}.

	Furthermore, as all of the existing Riemannian optimization packages require geometrical materials or specialized solvers, their supported numerical backends are also relatively limited. For example, although \Pkg{PyManopt} supports \Pkg{PyTorch}, \Pkg{Autograd} and \Pkg{Tensorflow} in its AD modules, its manifold module and solver module only support \Pkg{Numpy} as the numerical backend. Additionally, the packages \Pkg{GeoTorch}, \Pkg{Geoopt} and \Pkg{McTorch} are only developed based on \Pkg{PyTorch}. As a comparison, \CDOpt supports various of numerical backends in its modules, including the \Pkg{Numpy}, \Pkg{PyTorch}, \Pkg{JAX}, etc. We present a detailed comparison in Table \ref{Table:backends}.

	\begin{table}[htbp]
		\tiny
		\centering
		\caption{Comparison of the compared packages in terms of manifold and necessary materials for defining a new manifold.}
		\label{Table:manifold}
		\renewcommand\tabularxcolumn[1]{m{#1}}
		\begin{tabularx}{\linewidth}{p{3.5cm}XXXXX}
			\hline
			& \textbf{CDopt} & \textbf{Pymanopt} & \textbf{Geoopt} & \textbf{McTorch} & \textbf{GeoTorch}\\ \hline
			%			 \multicolumn{6}{|c|}{Predefined closed  manifolds}\\ \hline
			Euclidean space & \cmark &\cmark & \cmark & \cmark & \cmark  \\
			Sphere & \cmark &\cmark & \cmark & \cmark & \cmark  \\
			Oblique manifold & \cmark &\cmark & \xmark & \cmark & \xmark  \\
			Stiefel manifold & \cmark &\cmark & \cmark & \cmark & \cmark  \\
			Grassmann manifold & \cmark &\cmark & \cmark & \cmark & \cmark  \\
			Generalized Stiefel manifold & \cmark &\xmark & \xmark & \xmark & \xmark  \\
			Hyperbolic manifold & \cmark &\xmark & \xmark & \cmark & \xmark  \\
			Symplectic Stiefel manifold & \cmark &\xmark & \xmark & \xmark & \xmark  \\
			Complex sphere & \cmark &\cmark & \xmark & \xmark & \xmark  \\
			Complex oblique manifold & \cmark &\xmark & \xmark & \xmark & \xmark  \\
			Complex Stiefel manifold & \cmark &\xmark & \xmark & \xmark & \xmark  \\
			Product manifold & \cmark &\cmark & \cmark & \cmark & \cmark  \\  \hline
			Defining new manifolds 
			& The expression of $c(x)$
			&  Exponential and  logarithmic maps, retraction, vector transport,
			\codeobj{egrad2rgrad}, \codeobj{ehess2rhess}, inner product, distance, norm
			&  Exponential and logarithmic maps, retraction, vector transport,
			\codeobj{egrad2rgrad}, \codeobj{ehess2rhess}, inner product, distance, norm
			&  Exponential and logarithmic maps, retraction, vector transport,
			\codeobj{egrad2rgrad}, \codeobj{ehess2rhess}, inner product, distance, norm
			& The explicit formation of a surjection from $\bb{R}^N$ onto $\M$ for some $N > 0$.\\ \hline
		\end{tabularx}
	\end{table}

	\begin{table}[htbp]
		\scriptsize
		\centering
		\caption{Comparison of available solvers in general optimization (the optimization tasks defined by derivatives of the objective functions and manifolds), and training neural network via PyTorch.}
		\label{Table:solvers}
		\renewcommand\tabularxcolumn[1]{m{#1}}
		\begin{tabularx}{\linewidth}{lXXXXX}
			\hline
			& \textbf{CDopt} & \textbf{Pymanopt} & \textbf{Geoopt} & \textbf{McTorch} & \textbf{GeoTorch}\\ \hline
			Derivate-free solvers& 	NEWUOA, BOBYQA, LINCOA from \Pkg{PDFO} & N.A.& N.A.& N.A.& N.A. \\ \hline
			First-order  solvers 
			& Conjugate gradient, BFGS, limit-memory BFGS  from \Pkg{SciPy}, \Pkg{PyTorch-minimize}   and \Pkg{jaxopt} 
			& Riemannian conjugate gradient, Riemannian steepest descent, Riemannian line-search.
			& N.A.
			& N.A.
			& N.A. \\ \hline
			Second-order solvers & Newton-CG, truncated Newton, trust-region with Krylov subsolver, trust-region Newton-CG from \Pkg{SciPy} and \Pkg{PyTorch-minimize} 
			& Riemannian trust-region
			& N.A.
			& N.A.
			& N.A. \\ \hline
			Training in \Pkg{PyTorch} & All the optimizers from \Pkg{PyTorch.optim} and \Pkg{PyTorch-optimizer}
			& N.A. 
			& Riemannian Adam, Riemannian  line-search, Riemannian SGD 
			& Riemannian Adagrad, Riemannian SGD 
			& All the optimizers from \Pkg{PyTorch.optim} and \Pkg{PyTorch-optimizer}\\ \hline 
			Training in \Pkg{JAX} & All the optimizers from \Pkg{Optax} 
			& N.A.
			& N.A.
			& N.A.
			& N.A. \\ \hline
		\end{tabularx}
	\end{table}
	
	\begin{table}[htbp]
		\tiny
		\centering
		\caption{Comparison of the supported backends in different components of the compared packages.}
		\label{Table:backends}
		\renewcommand\tabularxcolumn[1]{m{#1}}
		\begin{tabularx}{\linewidth}{p{3cm}XXXXX}
			\hline
			& \textbf{CDopt} & \textbf{Pymanopt} & \textbf{Geoopt} & \textbf{McTorch} & \textbf{GeoTorch}\\ \hline
			Automatic differentiation & \Pkg{Autograd}, \Pkg{PyTorch}, \Pkg{JAX} & \Pkg{Autograd}, \Pkg{PyTorch}, \Pkg{TensorFlow}, \Pkg{Theano} & \Pkg{PyTorch} & \Pkg{PyTorch} & \Pkg{PyTorch}\\\hline
			Manifold & \Pkg{NumPy}, \Pkg{SciPy}, \Pkg{PyTorch}, \Pkg{JAX} & \Pkg{NumPy}, \Pkg{SciPy} & \Pkg{PyTorch} & \Pkg{PyTorch} & \Pkg{PyTorch} \\\hline
			Optimization & \Pkg{SciPy},\Pkg{PyTorch},  \Pkg{JAX}, \Pkg{Optax} & \Pkg{NumPy} & \Pkg{PyTorch} & \Pkg{PyTorch} & \Pkg{PyTorch}\\ \hline
		\end{tabularx}
	\end{table}

	\section{Examples}
	In this section, we briefly introduce the features of \CDOpt through several illustrative examples. Interested readers can refer to \url{https://cdopt.github.io/md_files/examples.html} for more examples of applying \CDOpt to solve Riemannian optimization problems.

	\subsection{Orthogonal dictionary learning}
	CDOpt is easy-to-use when it is employed to solve standard Riemannian optimization problems through its \codeobj{problem} class. 
	CDOpt integrates various numerical backends such as \Pkg{Autograd}, \Pkg{PyTorch}, and \Pkg{JAX}. Various important features of these packages, including AD packages, GPU supports and JIT compilation, are performed behind \CDOpt's \codeobj{problem} and \codeobj{manifold} modules. Therefore, users are only required to provide 
	a small amount of set-up codes for solving Riemannian optimization problems through \CDOpt.

	We briefly demonstrate the easy-to-use features of \CDOpt by a simple example on orthogonal dictionary learning (ODL). Given data $\{y_i\}_{i = 1,...,m}$ generated by $y_i = Q z_i$, where $Q$ is a fixed unknown orthogonal matrix and each $z_i$ follows an i.i.d. Bernoulli-Gaussian distribution with parameter $\theta \in (0,1)$. ODL aims to recover the matrix $Z = [z_1,...,z_m] \in \mathbb{R}^{m\times n}$ and the orthogonal matrix $Q \in \mathbb{R}^{n\times n}$ from the given data $Y = [y_1, ..., y_m]^\top \in \mathbb{R}^{m\times n}$.

	Based on the $\ell_4$-norm maximization model proposed in \cite{hu2020anefficiency,zhai2020complete}, we consider the following optimization problem over the Stiefel manifold:
	\begin{equation}
		\label{Example_ODL}
		\begin{aligned}
			\min_{X = [x_1,...x_n] \in \mathbb{R}^{n\times n}} \quad & f(X) := \sum_{1\leq i\leq m, 1\leq j\leq n} -(y_i^\top x_j)^4\\
			\text{s. t.} \quad & X^TX = I_n. 
		\end{aligned}
	\end{equation}
	The following script demonstrates how to solve \eqref{Example_ODL}  through \CDOpt and \Pkg{JAX}:
	\pythonexternal{./codes/example1.py}

	\subsection{Training neural network with orthogonal weight normalization}

	In this subsection, we present a simple example of training VGG19 \cite{simonyan2014very} with orthogonally constrained convolutional kernels.  The following script further demonstrates that neural networks with manifold constraints can be easily defined and trained through \CDOpt. Since all the neural layers from \codeobj{cdopt.nn} have the same APIs as those from \codeobj{torch.nn}, modifying a neural network to involve manifold constraints through \CDOpt is simple and requires minimum modifications to the standard \Pkg{PyTorch} codes.
	\pythonexternal{./codes/example2_vgg.py}

	\section{Numerical Experiments}
	In this section, we test the performance of \CDOpt on several important applications of \ref{Prob_Ori}. For standard Riemannian optimization problems, we compare the numerical performance of \CDOpt with \Pkg{PyManopt} (version 2.2.0). On the other hand, for training neural networks over Riemannian manifolds, we choose to use \Pkg{GeoTorch} (version 0.3.0) and \Pkg{McTorch} (version 0.1.0) in the comparison.  All the experiments are performed on Ubuntu 20.02 platform with Intel Xeon 6342 and NVIDIA GeForce RTX 3090.  Moreover, we run all these numerical experiments in Python 3.8.10 with \Pkg{NumPy} 1.23, \Pkg{SciPy} 1.8.1, \Pkg{PyTorch} 1.9.0, \Pkg{CUDA} 11.4 and \Pkg{JAX} 0.3.14.

	\begin{rmk}
		It is worth mentioning that the limited memory BFGS (L-BFGS) and conjugate gradient (CG) solvers in \Pkg{SciPy} package are implemented in FORTRAN and C++, respectively, hence exhibiting high computational efficiency. These two solvers are integrated into \Pkg{SciPy} via \Pkg{F2PY} interfaces. In contrast, all the solvers in \Pkg{PyManopt} are developed entirely in Python. 
		
		This distinction highlights one of the major advantages of the \CDOpt package: \CDOpt admits direct implementations of highly efficient unconstrained solvers for Riemannian optimization problems. Nonetheless, due to the inherent differences in the implementation languages of the solvers in \CDOpt and \Pkg{PyManopt}, a direct comparison of CPU times might not be equitable for the solvers in \Pkg{PyManopt}. Therefore, for better illustrations of the numerical performance of \CDOpt, we report both the CPU time and the number of evaluations of function values and gradients in our numerical experiments.
	\end{rmk}

	\subsection{Choices of different constraint dissolving mappings}
	
	Section 4.1 demonstrates that \CDOpt includes several predefined well-established Riemannian manifolds.
	Additionally, \CDOpt supports the customized definition for manifolds, which only requires the expression of the constraints function $c$. Then \CDOpt automatically generates all other essential components using \eqref{Eq_general_cd_mapping} and AD algorithms.
	
	The predefined manifolds listed in Table \eqref{Table_cdopt_manifold} require only matrix-matrix multiplications to compute constraint dissolving mappings. In contrast, when employing $\A_c$ from \eqref{Eq_general_cd_mapping} as the constraint dissolving mapping, one must solve $p\times p$ linear equations to obtain $\A_c(x)$. Intuitively, for the predefined manifolds in Table \eqref{Table_cdopt_manifold}, it is generally more computationally efficient to use the mappings listed in Table \ref{Table_implementation} rather than those determined by   \eqref{Eq_general_cd_mapping}.

	In this subsection, we aim to perform preliminary numerical experiments to evaluate the computational efficiency of these two different choices of  constraint dissolving mappings in \CDOpt. In our numerical experiments, we choose the sphere, Stiefel manifold, and symplectic Stiefel manifold for the comparison. For these manifolds, with a given randomly generated $x \in \bb{R}^n$, we measure the CPU time for computing the $\A(x)$, $\Ja(x)$, and $\DJa(x)$ with respect to different choices of the constraint dissolving mappings from Table \ref{Table_implementation} and \eqref{Eq_general_cd_mapping}, respectively. Each test instance is replicated for $1000$ times by using the built-in \texttt{timeit} module in Python, and we report the average CPU time in seconds.

	Figure \ref{Fig_differentA} shows the comparison of the computation time for different constraint dissolving mappings in \CDOpt. From the curves in Figure \ref{Fig_differentA}, it can be inferred that the computational time for the mappings $\A$, $\Ja$, and $\DJa$ in Table \ref{Table_implementation}, is considerably shorter compared to the cases where $\A$ is selected based on \eqref{Eq_general_cd_mapping}. Consequently, for any Riemannian optimization problem where the manifold is listed in Table \ref{Table_cdopt_manifold}, the predefined manifolds in CDOpt are recommended for better computational efficiency.

	\begin{figure}[!htbp]
		\centering
		\subfigure[Sphere]{
			\begin{minipage}[t]{0.33\linewidth}
				\centering
				\includegraphics[width=\linewidth]{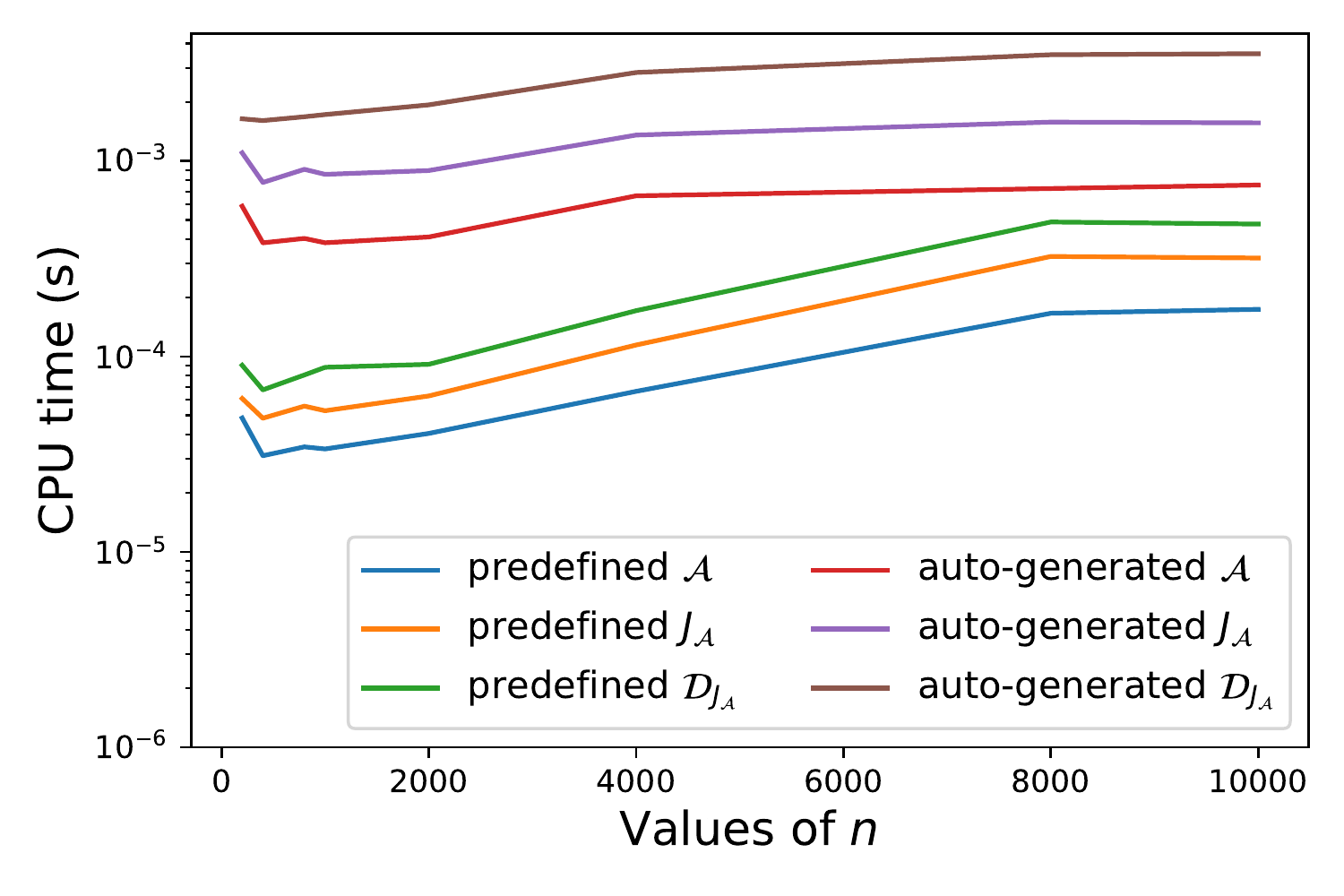}
				\label{Fig:Fig_differentA_Stiefel_p1}
				%\caption{Problem 2, SLAM}
			\end{minipage}%
		}%
		\subfigure[Stiefel with $s = 10$]{
			\begin{minipage}[t]{0.33\linewidth}
				\centering
				\includegraphics[width=\linewidth]{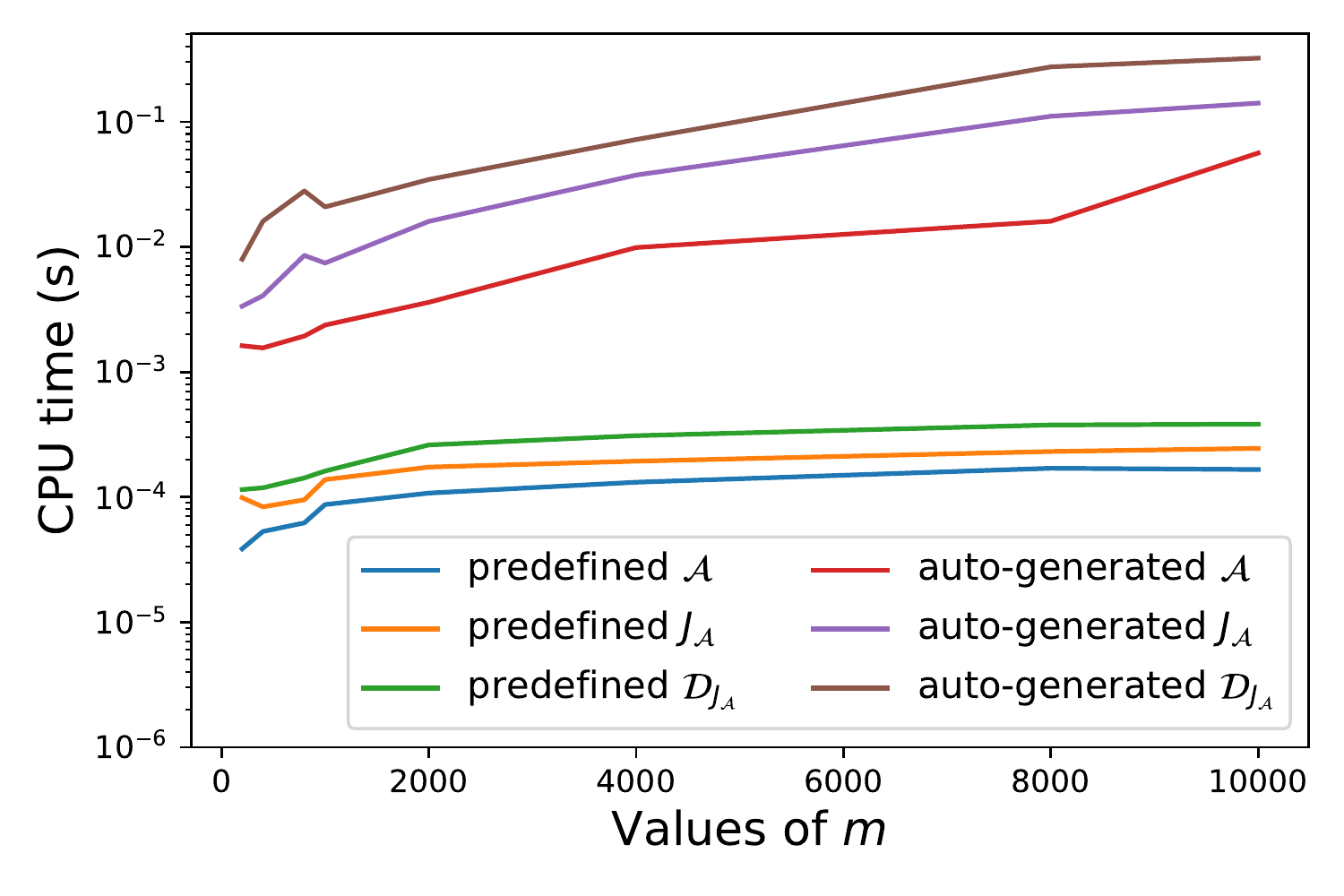}
				\label{Fig:Fig_differentA_Stiefel_p10}
				%\caption{Problem 2, SLAM}
			\end{minipage}%
		}%
		\subfigure[Symplectic Stiefel with $s = 10$]{
			\begin{minipage}[t]{0.33\linewidth}
				\centering
				\includegraphics[width=\linewidth]{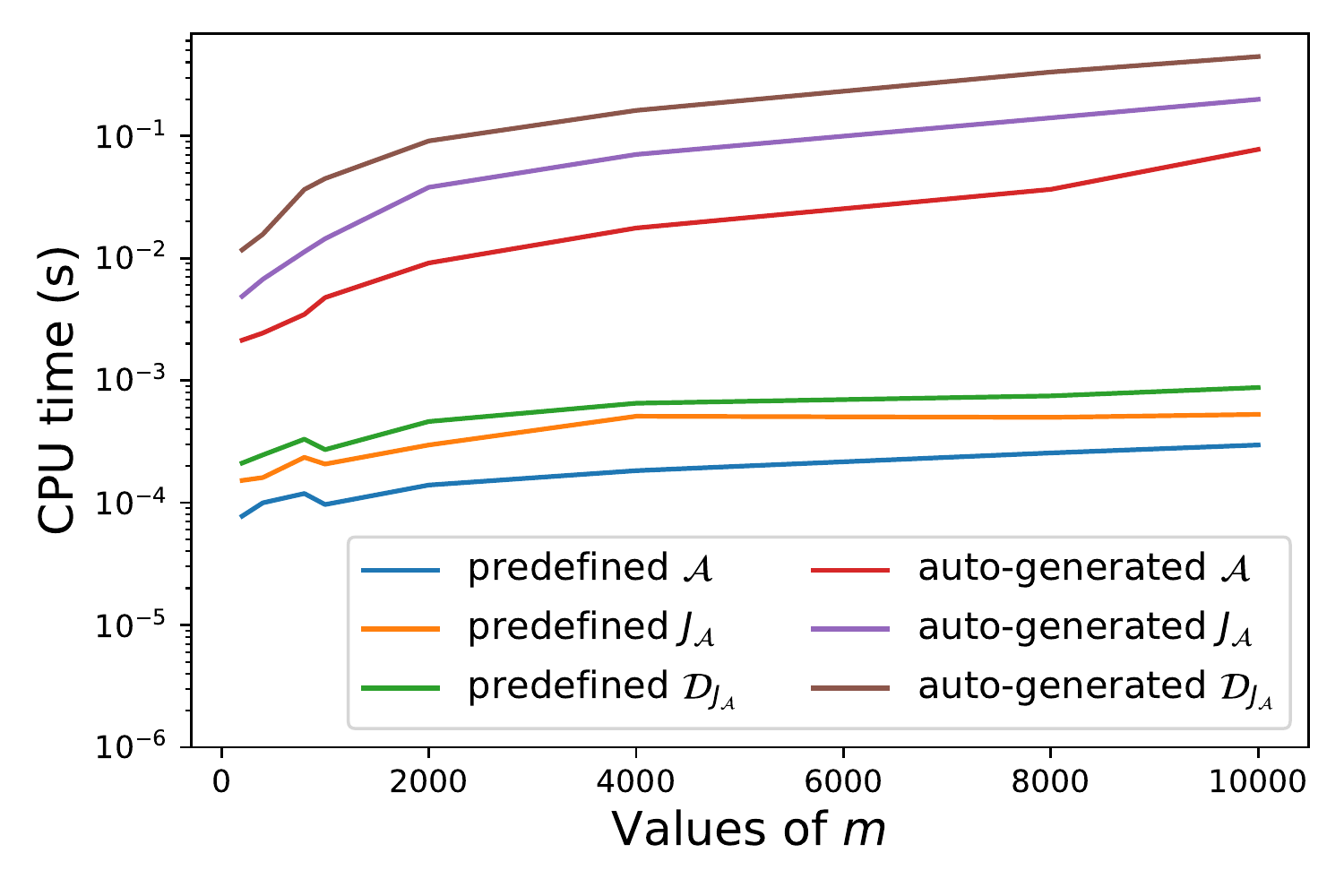}
				\label{Fig:Fig_differentA_sympStiefel_p10}
				%\caption{Problem 2, SLAM}
			\end{minipage}%
		}%
		\caption{Comparisons of the computational times for different choices of the constraint dissolving mappings in \CDOpt. Here ``predefined'' refers to choosing $\A$ as specified in Table \ref{Table_implementation}, and ``auto-generated'' means choosing $\A$ by \eqref{Eq_general_cd_mapping}.} 
		\label{Fig_differentA}
	\end{figure}

	\subsection{Nearest correlation matrix}
	\label{Subsection_Numerical_NCM}
	In this subsection, we test the numerical performance of \CDOpt on optimization problems defined by its \codeobj{problem} interface. Moreover, we compare \CDOpt with the state-of-the-art Riemannian optimization package \Pkg{PyManopt} (version 2.2.0), where the Riemannian optimization problems must be defined through its \codeobj{problem} interface. 
	
	Our test example is the problem of finding the nearest low-rank correlation matrix (NCM) to a given matrix $G \in \bb{R}^{n\times p}$, which can be reshaped as an optimization problem over the oblique manifold
	\begin{equation}
		\label{Example_NCM}
		\begin{aligned}
			\min_{X \in \bb{R}^{n\times p}} \quad & f(X) = \frac{1}{2} \norm{ H\circ(XX\tp - G)}_F^2\\
			\text{s.t.} \quad & \mathrm{Diag}(XX\tp) = I_n,
		\end{aligned}
	\end{equation}
	where $H \in \bb{R}^{n\times n}$ is a nonnegative weight matrix. For all the numerical experiments in this subsection, we generate the matrix $\hat{G}$ from the gene expression data provided in \cite{li2010inexact}\footnote{available at \url{https://blog.nus.edu.sg/mattohkc/files/2019/11/covseldata.zip}}. Then the matrix $G$ is generated by perturbing $\hat{G}$  by $G = (1-\theta)\hat{G} + \theta E$, where $\theta \geq 0$ is a prefixed parameter, and $E \in \bb{R}^{n\times n}$ is a randomly generated matrix with all of its diagonal entry equals to $1$. The weight matrix $H$ in \eqref{Example_NCM} is chosen as a symmetric matrix whose entries are uniformly distributed in $[0, 1]$.  
	
	We set the penalty parameter $\beta$ for \ref{Prob_Pen} by setting \codeobj{beta = 'auto'} in the instantiation of the \codeobj{problem} class in \CDOpt. We choose the L-BFGS, CG and trust-region solvers from \Pkg{SciPy} to test the numerical performance of \CDOpt. For comparison, we choose the Riemannian CG solver and Riemannian trust-region solver from \Pkg{PyManopt} since L-BFGS is not available in \Pkg{PyManopt}. For \CDOpt, we stop these solvers when the maximum number of iterations exceeds $10000$ or the norm of the gradient its corresponding constraint dissolving function is less than $10^{-5}$. For the solvers from \Pkg{PyManopt}, we stop these solvers once the norm of its Riemannian gradient is smaller than $10^{-5}$, or the maximum number of iterations exceeds $10000$. All the other parameters are fixed as their default values. Additionally, all the solvers start from the same initial point, which is randomly generated on the oblique manifold in each test instance. 
	
	Furthermore, we fix the numerical backends as the \Pkg{PyTorch} package in both \CDOpt and \Pkg{PyManopt}. We only provide the expression of the objective function in all the test instances, thus the gradient and the Hessian of $f$ are automatically computed by the AD packages from the \Pkg{PyTorch} backends in \CDOpt or \Pkg{PyManopt}. 
	
	Table \ref{Table_NCM_ER_CPU} -- Table \ref{Table_NCM_Lymph_GPU} report the numerical results on the comparison of \CDOpt and \Pkg{PyManopt}, where the detailed meanings of their column headers are presented in Table \ref{Table_header}. It is worth mentioning that in  \Pkg{PyManopt} (version 2.2.0), the number of evaluations for gradients and Hessians is unavailable in its built-in Riemannian trust-region solver (pymanopt-rtr). Therefore, those entries are recorded as ``None'' in Table \ref{Table_NCM_ER_CPU} - Table \ref{Table_NCM_Lymph_GPU}. 
	
	\begin{table}[htbp]
		\centering
		
		\caption{The detailed meanings of the column headers in the tables that report the results of numerical experiments. Here $x^*$ denoted the output of the solvers. }
		\label{Table_header}
		\begin{tabular}{l|l}
			\hline
			Abbreviation & Detailed meaning \\ \hline
			fval        &   The value of $f(x^*)$     \\ \hline
			iter        &   Total number of iterations \\ \hline
			eval\_f      & Total number of evaluations of the function values of $f$ \\ \hline 
			eval\_grad   & Total number of evaluations of the gradients of $f$ \\ \hline 
			stat        & The $\ell_2$-norm of the Riemannian gradient of $f$ at $x^*$ \\\hline   
			feas        & The value of $\norm{c(x^*)}$ \\ \hline 
			time (s)    & Total running time of the solver (in seconds) \\\hline
			acc         & Test accuracy at the final epoch \\ \hline
			time/epoch  & The averaged running time (in seconds) per epoch \\ \hline
		\end{tabular}
	\end{table}

	From Table \ref{Table_NCM_ER_CPU} -- Table \ref{Table_NCM_Lymph_GPU}, we can observe that \CDOpt achieves comparable performance as \Pkg{PyManopt} in all the numerical experiments. Moreover, although \CDOpt needs to reshape the variables in each iteration, applying the conjugate gradient method and trust-region method from \Pkg{SciPy} package shows superior performance over their Riemannian counterparts from \Pkg{PyManopt} when $p$ exceeds $100$, particularly on GPU architectures. Furthermore, benefiting from the highly efficient L-BFGS solver that is programmed in FORTRAN and wrapped by the \Pkg{SciPy} package, \CDOpt gains significant advantages against the compared Riemannian solvers from \Pkg{PyManopt}. These numerical results indicate the great potential of our developed \CDOpt package.

	\begin{table}[htbp]
		\centering
		\tiny
		\caption{Nearest correlation matrix problem on the ER dataset on CPUs. }
		\label{Table_NCM_ER_CPU}
		\begin{tabular}{l|l|ccccccc}
			\hline
			\multicolumn{2}{c|}{}                  & fval & iter & eval\_f & eval\_grad & stat & feas & time (s)  \\ \hline
			\multirow{5}{*}{$p = 10$}   
			& cdopt-lbfgs & 2.30e+03  & 178  & 188 & 188  & 9.93e-06     & 4.50e-15     & 0.63 \\
			& cdopt-cg & 2.30e+03  & 238  & 363 & 363  & 9.26e-06     & 4.29e-15     & 0.96 \\
			& cdopt-tr & 2.30e+03  & 75 & 76 & 71  & 7.88e-06     & 4.35e-15     & 4.06 \\\cline{2-9}
			& pymanopt-rcg   & 2.30e+03  & 314  & 314  & 314    & 9.56e-06     & 5.42e-15     & 0.74 \\
			& pymanopt-rtr  & 2.30e+03  & 41  & None  & None    & 2.75e-06     & 5.54e-15     & 1.42 \\ \hline
			\multirow{5}{*}{$p = 20$}   
			& cdopt-lbfgs & 7.69e+02  & 196  & 205 & 205  & 8.79e-06     & 4.62e-15     & 0.83 \\
			& cdopt-cg & 7.69e+02  & 250  & 388 & 388  & 9.44e-06     & 2.78e-13     & 7.56 \\
			& cdopt-tr & 7.69e+02  & 63 & 64 & 58  & 2.54e-07     & 3.92e-15     & 40.37 \\ \cline{2-9}
			& pymanopt-rcg   & 7.69e+02  & 301  & 301  & 301    & 8.66e-06     & 6.01e-15     & 6.40 \\
			& pymanopt-rtr  & 7.69e+02  & 26  & None  & None    & 3.02e-08     & 6.14e-15     & 10.38 \\\hline
			\multirow{5}{*}{$p = 30$}   
			& cdopt-lbfgs & 4.27e+02  & 323  & 337 & 337  & 8.45e-06     & 3.55e-15     & 2.03 \\
			& cdopt-cg & 4.27e+02  & 732  & 1089 & 1089  & 8.22e-06     & 5.63e-15     & 24.95 \\
			& cdopt-tr & 4.27e+02  & 73 & 74 & 66  & 9.48e-06     & 4.88e-15     & 36.32 \\ \cline{2-9}
			& pymanopt-rcg   & 4.27e+02  & 599  & 599  & 599    & 7.01e-06     & 5.58e-15     & 12.24 \\
			& pymanopt-rtr  & 4.27e+02  & 31  & None  & None    & 1.64e-07     & 6.45e-15     & 29.63 \\\hline
			\multirow{5}{*}{$p = 50$} 
			& cdopt-lbfgs & 2.34e+02  & 428  & 445 & 445  & 7.97e-06     & 4.61e-13     & 4.86 \\
			& cdopt-cg & 2.34e+02  & 732  & 1107 & 1107  & 6.31e-06     & 4.09e-15     & 26.19 \\
			& cdopt-tr & 2.34e+02  & 59 & 60 & 52  & 5.02e-07     & 4.73e-15     & 61.21 \\\cline{2-9}
			& pymanopt-rcg   & 2.34e+02  & 1175  & 1175  & 1175    & 9.98e-06     & 5.68e-15     & 27.56 \\
			& pymanopt-rtr  & 2.34e+02  & 34  & None  & None    & 8.11e-08     & 5.68e-15     & 57.05 \\\hline
			\multirow{5}{*}{$p = 100$} 
			& cdopt-lbfgs & 1.48e+02  & 561  & 579 & 579  & 8.89e-06     & 1.54e-12     & 6.28 \\
			& cdopt-cg & 1.48e+02  & 1256  & 1875 & 1875  & 9.61e-06     & 3.99e-15     & 64.08 \\
			& cdopt-tr & 1.48e+02  & 94 & 95 & 81  & 5.51e-08     & 4.16e-15     & 209.16 \\ \cline{2-9}
			& pymanopt-rcg   & 1.48e+02  & 1535  & 1535  & 1535    & 9.45e-06     & 5.67e-15     & 57.02 \\
			& pymanopt-rtr  & 1.48e+02  & 44  & None  & None    & 6.09e-08     & 5.54e-15     & 405.36 \\\hline
			\multirow{5}{*}{$p = 200$} 
			& cdopt-lbfgs & 1.30e+02  & 684  & 704 & 704  & 6.69e-06     & 5.37e-12     & 12.59 \\
			& cdopt-cg & 1.30e+02  & 1157  & 1751 & 1751  & 8.68e-06     & 8.40e-14     & 74.15 \\
			& cdopt-tr & 1.30e+02  & 64 & 65 & 57  & 5.77e-07     & 4.54e-15     & 224.90 \\\cline{2-9}
			& pymanopt-rcg   & 1.30e+02  & 1738  & 1738  & 1738    & 9.76e-06     & 4.73e-15     & 89.65 \\
			& pymanopt-rtr  & 1.30e+02  & 41  & None  & None    & 1.93e-07     & 4.75e-15     & 375.63 \\\hline
		\end{tabular}
	\end{table}

	\begin{table}[htbp]
		\centering
		\tiny
		\caption{Nearest correlation matrix problem on the ER dataset on GPUs. }
		\label{Table_NCM_ER_GPU}
		\begin{tabular}{l|l|ccccccc}
			\hline
			\multicolumn{2}{c|}{}                  & fval & iter & eval\_f & eval\_grad & stat & feas & time (s)  \\ \hline
			\multirow{5}{*}{$p = 10$}   
			& cdopt-lbfgs & 2.30e+03  & 182  & 192 & 192  & 9.95e-06     & 1.93e-14     & 0.47 \\
			& cdopt-cg & 2.30e+03  & 230  & 346 & 346  & 8.40e-06     & 5.68e-15     & 0.72 \\
			& cdopt-tr & 2.30e+03  & 75 & 76 & 71  & 2.74e-06     & 4.09e-15     & 2.23 \\\cline{2-9}
			& pymanopt-rcg   & 2.30e+03  & 314  & 314  & 314    & 9.56e-06     & 5.42e-15     & 0.74 \\
			& pymanopt-rtr  & 2.30e+03  & 41  & None  & None    & 2.75e-06     & 5.54e-15     & 1.42 \\ \hline
			\multirow{5}{*}{$p = 20$}   
			& cdopt-cg & 7.69e+02  & 256  & 374 & 374  & 9.39e-06     & 2.95e-14     & 1.36 \\
			& cdopt-lbfgs & 7.69e+02  & 198  & 207 & 207  & 6.43e-06     & 4.49e-15     & 0.64 \\
			& cdopt-tr & 7.69e+02  & 65 & 66 & 60  & 5.87e-06     & 4.12e-15     & 5.22 \\ \cline{2-9}
			& pymanopt-rcg   & 7.69e+02  & 299  & 299  & 299    & 9.77e-06     & 5.74e-15     & 1.36 \\
			& pymanopt-rtr  & 7.69e+02  & 26  & None  & None    & 2.98e-08     & 5.74e-15     & 5.51 \\\hline
			\multirow{5}{*}{$p = 30$}  
			& cdopt-lbfgs & 4.27e+02  & 318  & 328 & 328  & 9.92e-06     & 2.99e-15     & 1.31 \\
			& cdopt-cg & 4.27e+02  & 731  & 1078 & 1078  & 6.17e-06     & 3.31e-14     & 4.27 \\
			& cdopt-tr & 4.27e+02  & 78 & 79 & 71  & 2.05e-07     & 4.37e-15     & 6.30 \\ \cline{2-9}
			& pymanopt-rcg   & 4.27e+02  & 711  & 711  & 711    & 8.85e-06     & 5.70e-15     & 5.16 \\
			& pymanopt-rtr  & 4.27e+02  & 31  & None  & None    & 1.59e-07     & 6.27e-15     & 8.06 \\\hline
			\multirow{5}{*}{$p = 50$} 
			& cdopt-lbfgs & 2.34e+02  & 422  & 436 & 436  & 8.47e-06     & 1.60e-13     & 2.16 \\
			& cdopt-cg & 2.34e+02  & 642  & 948 & 948  & 6.80e-06     & 5.35e-15     & 3.64 \\
			& cdopt-tr & 2.34e+02  & 74 & 75 & 64  & 2.04e-06     & 4.50e-15     & 8.57 \\\cline{2-9}
			& pymanopt-rcg   & 2.34e+02  & 1194  & 1194  & 1194    & 7.26e-06     & 5.48e-15     & 7.24 \\
			& pymanopt-rtr  & 2.34e+02  & 34  & None  & None    & 4.58e-08     & 5.47e-15     & 23.41 \\\hline
			\multirow{5}{*}{$p = 100$} 
			& cdopt-lbfgs & 1.48e+02  & 575  & 595 & 595  & 8.16e-06     & 1.45e-13     & 4.56 \\
			& cdopt-cg & 1.48e+02  & 1045  & 1550 & 1550  & 9.90e-06     & 6.23e-14     & 8.59 \\
			& cdopt-tr & 1.48e+02  & 91 & 92 & 81  & 5.27e-08     & 4.29e-15     & 25.19 \\\cline{2-9}
			& pymanopt-rcg   & 1.48e+02  & 2040  & 2040  & 2040    & 9.49e-06     & 5.76e-15     & 20.34 \\
			& pymanopt-rtr  & 1.48e+02  & 44  & None  & None    & 6.24e-08     & 5.56e-15     & 195.45 \\\hline
			\multirow{5}{*}{$p = 200$} 
			& cdopt-lbfgs & 1.30e+02  & 671  & 701 & 701  & 9.55e-06     & 1.88e-11     & 10.67 \\
			& cdopt-cg & 1.30e+02  & 1231  & 1856 & 1856  & 8.20e-06     & 3.47e-15     & 17.57 \\
			& cdopt-tr & 1.30e+02  & 67 & 68 & 58  & 1.57e-07     & 4.40e-15     & 36.79 \\\cline{2-9}
			& pymanopt-rcg   & 1.30e+02  & 1794  & 1794  & 1794    & 9.13e-06     & 4.88e-15     & 34.44 \\
			& pymanopt-rtr  & 1.30e+02  & 41  & None  & None    & 1.46e-07     & 4.78e-15     & 186.64 \\\hline
		\end{tabular}
	\end{table}

	\begin{table}[htbp]
		\centering
		\tiny
		\caption{Nearest correlation matrix problem on the Lymph dataset on CPUs. }
		\label{Table_NCM_Lymph_CPU}
		\begin{tabular}{l|l|ccccccc}
			\hline
			\multicolumn{2}{c|}{}                  & fval & iter & eval\_f & eval\_grad & stat & feas & time (s)  \\ \hline
			\multirow{5}{*}{$p = 10$}   
			& cdopt-lbfgs & 1.93e+03  & 161  & 168 & 168  & 9.28e-06     & 2.62e-14     & 0.51 \\
			& cdopt-cg & 1.93e+03  & 223  & 333 & 333  & 8.21e-06     & 4.25e-15     & 0.78 \\
			& cdopt-tr & 1.93e+03  & 56 & 57 & 51  & 2.23e-06     & 4.19e-15     & 2.64 \\\cline{2-9}
			& pymanopt-rcg   & 1.93e+03  & 248  & 248  & 248    & 9.32e-06     & 5.25e-15     & 0.87 \\
			& pymanopt-rtr  & 1.93e+03  & 34  & None  & None    & 4.75e-10     & 5.37e-15     & 1.54 \\ \hline
			\multirow{5}{*}{$p = 20$}  
			& cdopt-lbfgs & 6.55e+02  & 217  & 226 & 226  & 9.55e-06     & 3.77e-15     & 0.89 \\
			& cdopt-cg & 6.55e+02  & 349  & 524 & 524  & 9.28e-06     & 1.79e-14     & 6.41 \\
			& cdopt-tr & 6.55e+02  & 60 & 61 & 58  & 2.23e-06     & 4.43e-15     & 20.18 \\\cline{2-9}
			& pymanopt-rcg   & 6.55e+02  & 462  & 462  & 462    & 9.22e-06     & 4.90e-15     & 11.64 \\
			& pymanopt-rtr  & 6.55e+02  & 31  & None  & None    & 2.07e-09     & 6.04e-15     & 17.35 \\\hline
			\multirow{5}{*}{$p = 30$}   
			& cdopt-lbfgs & 3.56e+02  & 351  & 364 & 364  & 9.97e-06     & 3.78e-15     & 1.80 \\
			& cdopt-cg & 3.56e+02  & 443  & 659 & 659  & 9.29e-06     & 3.98e-15     & 7.52 \\
			& cdopt-tr & 3.56e+02  & 69 & 70 & 63  & 3.09e-07     & 4.06e-15     & 32.78 \\\cline{2-9}
			& pymanopt-rcg   & 3.56e+02  & 701  & 701  & 701    & 9.37e-06     & 5.43e-15     & 14.49 \\
			& pymanopt-rtr  & 3.56e+02  & 50  & None  & None    & 6.16e-07     & 5.58e-15     & 79.33 \\\hline
			\multirow{5}{*}{$p = 50$} 
			& cdopt-lbfgs & 1.80e+02  & 374  & 389 & 389  & 9.64e-06     & 3.08e-15     & 2.40 \\
			& cdopt-cg & 1.80e+02  & 738  & 1107 & 1107  & 6.52e-06     & 3.65e-15     & 29.03 \\
			& cdopt-tr & 1.80e+02  & 60 & 61 & 55  & 9.36e-06     & 5.16e-15     & 40.50 \\\cline{2-9}
			& pymanopt-rcg   & 1.80e+02  & 993  & 993  & 993    & 9.97e-06     & 5.28e-15     & 25.93 \\
			& pymanopt-rtr  & 1.80e+02  & 46  & None  & None    & 3.40e-06     & 6.31e-15     & 42.89 \\\hline
			\multirow{5}{*}{$p = 100$} 
			& cdopt-lbfgs & 1.03e+02  & 628  & 651 & 651  & 9.53e-06     & 1.51e-13     & 6.10 \\
			& cdopt-cg & 1.03e+02  & 1437  & 2155 & 2155  & 9.95e-06     & 4.52e-15     & 69.45 \\
			& cdopt-tr & 1.03e+02  & 89 & 90 & 76  & 1.79e-07     & 4.02e-15     & 170.99 \\\cline{2-9}
			& pymanopt-rcg   & 1.03e+02  & 2763  & 2763  & 2763    & 8.92e-06     & 5.33e-15     & 97.32 \\
			& pymanopt-rtr  & 1.03e+02  & 45  & None  & None    & 6.17e-07     & 5.18e-15     & 140.68 \\\hline
			\multirow{5}{*}{$p = 200$} 
			& cdopt-lbfgs & 8.84e+01  & 522  & 540 & 540  & 8.97e-06     & 1.14e-11     & 8.44 \\
			& cdopt-cg & 8.84e+01  & 716  & 1074 & 1074  & 9.94e-06     & 1.87e-13     & 42.01 \\
			& cdopt-tr & 8.84e+01  & 58 & 59 & 49  & 2.71e-07     & 4.63e-15     & 155.79 \\\cline{2-9}
			& pymanopt-rcg   & 8.84e+01  & 2299  & 2299  & 2299    & 9.95e-06     & 4.35e-15     & 113.34 \\
			& pymanopt-rtr  & 8.84e+01  & 41  & None  & None    & 2.51e-08     & 4.53e-15     & 388.71 \\\hline
		\end{tabular}
	\end{table}

	\begin{table}[htbp]
		\centering
		\tiny
		\caption{Nearest correlation matrix problem on the Lymph dataset on GPUs. }
		\label{Table_NCM_Lymph_GPU}
		\begin{tabular}{l|l|ccccccc}
			\hline
			\multicolumn{2}{c|}{}                  & fval & iter & eval\_f & eval\_grad & stat & feas & time (s)  \\ \hline
			\multirow{5}{*}{$p = 10$}   
			& cdopt-lbfgs & 1.93e+03  & 156  & 162 & 162  & 9.65e-06     & 7.00e-15     & 0.39 \\
			& cdopt-cg & 1.93e+03  & 208  & 320 & 320  & 9.59e-06     & 2.20e-14     & 0.56 \\
			& cdopt-tr & 1.93e+03  & 55 & 56 & 50  & 6.21e-08     & 3.50e-15     & 1.74 \\\cline{2-9}
			& pymanopt-rcg   & 1.93e+03  & 248  & 248  & 248    & 9.36e-06     & 4.68e-15     & 0.48 \\
			& pymanopt-rtr  & 1.93e+03  & 34  & None  & None    & 3.73e-10     & 5.23e-15     & 1.11 \\\hline
			\multirow{5}{*}{$p = 20$}
			& cdopt-lbfgs & 6.55e+02  & 223  & 230 & 230  & 8.31e-06     & 3.85e-15     & 0.65 \\
			& cdopt-cg & 6.55e+02  & 391  & 596 & 596  & 8.35e-06     & 2.97e-15     & 2.48 \\
			& cdopt-tr & 6.55e+02  & 57 & 58 & 56  & 7.74e-06     & 4.04e-15     & 3.10 \\\cline{2-9}
			& pymanopt-rcg   & 6.55e+02  & 511  & 511  & 511    & 7.61e-06     & 5.01e-15     & 2.68 \\
			& pymanopt-rtr  & 6.55e+02  & 31  & None  & None    & 2.20e-09     & 5.99e-15     & 7.40 \\\hline
			\multirow{5}{*}{$p = 30$}   
			& cdopt-lbfgs & 3.56e+02  & 357  & 373 & 373  & 8.49e-06     & 4.08e-15     & 1.20 \\
			& cdopt-cg & 3.56e+02  & 463  & 699 & 699  & 6.43e-06     & 3.70e-15     & 2.29 \\
			& cdopt-tr & 3.56e+02  & 69 & 70 & 63  & 3.21e-07     & 3.41e-15     & 4.99 \\\cline{2-9}
			& pymanopt-rcg   & 3.56e+02  & 614  & 614  & 614    & 8.79e-06     & 5.28e-15     & 5.38 \\
			& pymanopt-rtr  & 3.56e+02  & 50  & None  & None    & 6.18e-07     & 5.70e-15     & 31.51 \\\hline
			\multirow{5}{*}{$p = 50$}
			& cdopt-lbfgs & 1.80e+02  & 374  & 388 & 388  & 8.40e-06     & 3.58e-15     & 1.77 \\
			& cdopt-cg & 1.80e+02  & 681  & 1005 & 1005  & 9.15e-06     & 2.38e-13     & 4.70 \\
			& cdopt-tr & 1.80e+02  & 60 & 61 & 55  & 9.46e-06     & 5.62e-15     & 7.74 \\\cline{2-9}
			& pymanopt-rcg   & 1.80e+02  & 1130  & 1130  & 1130    & 7.61e-06     & 5.20e-15     & 8.27 \\
			& pymanopt-rtr  & 1.80e+02  & 46  & None  & None    & 3.08e-06     & 6.03e-15     & 17.18 \\\hline
			\multirow{5}{*}{$p = 100$}
			& cdopt-lbfgs & 1.03e+02  & 641  & 657 & 657  & 7.87e-06     & 2.76e-14     & 4.75 \\
			& cdopt-cg & 1.03e+02  & 1615  & 2417 & 2417  & 9.72e-06     & 4.83e-15     & 12.77 \\
			& cdopt-tr & 1.03e+02  & 88 & 89 & 75  & 4.76e-06     & 4.52e-15     & 22.09 \\\cline{2-9}
			& pymanopt-rcg   & 1.03e+02  & 2734  & 2734  & 2734    & 9.20e-06     & 5.33e-15     & 28.66 \\
			& pymanopt-rtr  & 1.03e+02  & 45  & None  & None    & 1.01e-06     & 5.26e-15     & 66.00 \\\hline
			\multirow{5}{*}{$p = 200$}
			& cdopt-lbfgs & 8.84e+01  & 519  & 537 & 537  & 9.32e-06     & 3.92e-12     & 7.19 \\
			& cdopt-cg & 8.84e+01  & 611  & 923 & 923  & 8.56e-06     & 9.29e-15     & 7.33 \\
			& cdopt-tr & 8.84e+01  & 56 & 57 & 49  & 5.78e-06     & 8.23e-15     & 26.17 \\\cline{2-9}
			& pymanopt-rcg   & 8.84e+01  & 1937  & 1937  & 1937    & 9.03e-06     & 4.34e-15     & 33.71 \\
			& pymanopt-rtr  & 8.84e+01  & 41  & None  & None    & 2.52e-08     & 4.38e-15     & 201.75 \\\hline
		\end{tabular} 
	\end{table}

	\subsection{Nonlinear eigenvalue problem}
	\label{Subsection_Numerical_NEP}
	In this subsection, we test the numerical performance of the compared packages on solving the following smooth optimization problem over the Stiefel manifold:
	\begin{equation}\label{neg}
		\begin{aligned}
			\min_{X \in \bb{R}^{n\times p}} \quad &f(X) = \frac{1}{2} \tr\left( X\tp LX \right) + \frac{\alpha}{4} \rho(X)\tp L^{\dagger} \rho(X)\\
			\text{s. t.} \quad & X\tp X = I_p,
		\end{aligned}
	\end{equation}
	where $\rho(X):= \mathrm{diag}(XX\tp)$, $L$ is a tridiagonal matrix with $2$ as its diagonal entries and $-1$ as its subdiagonal and superdiagonal entries. Moreover, $L^\dagger$ refers to the pseudo-inverse of $L$. Problem \eqref{neg}, arisen from electronic structure calculation, has become a standard testing problem for investigating the convergence of self-consistent field methods due to its simplicity \cite{lin2013elliptic}.

	In all the numerical experiments of this subsection, we provide the expression of the function value, the gradient, and the Hessian-vector product, and set the numerical backends for \CDOpt and \Pkg{PyManopt} as \Pkg{Numpy}. 
	We choose the L-BFGS, CG, and trust-region solvers from \Pkg{SciPy} to test the numerical performance of \CDOpt. For comparison, we choose the Riemannian CG solver and Riemannian trust-region solver from \Pkg{PyManopt} since L-BFGS is not available in \Pkg{PyManopt}. All the hyperparameters of these solvers are fixed by the same settings in Scetion \ref{Subsection_Numerical_NCM}. Furthermore, following the techniques in \cite{xiao2021solving}, we first randomly generate a reference point $\tilde{X} \in \ca{S}_{n,p}$, and generate the initial point $X_{init}$ as the eigenvectors that corresponds to the largest $p$ eigenvalues of $L + \alpha \cdot\mathrm{Diag}( L^\dagger \rho(\tilde{X}) )$. 
	
	Table \ref{Table_NEP_n1000} and Table \ref{Table_NEP_n3000} 
	present the numerical results on the comparison of 
	\CDOpt and \Pkg{PyManopt}, and the detailed meanings of their column headers are 
	given in Table \ref{Table_header}. From Table \ref{Table_NEP_n1000} and Table \ref{Table_NEP_n3000}, we can observe that \CDOpt achieves the same function values in comparable iterations as \Pkg{PyManopt} in all the tested instances. Moreover, \CDOpt shows comparable performance in the aspect of CPU time for large-scale problems, especially when $p$ is large.

	\begin{table}[htbp]
		\centering
		\tiny
		\caption{Test results on nonlinear eigenvalue problems with $n = 1000$. }
		\label{Table_NEP_n1000}
		\begin{tabular}{l|l|ccccccc}
			\hline
			\multicolumn{2}{c|}{}                  & fval & iter & eval\_f & eval\_grad & stat & feas & time (s)  \\ \hline
			\multirow{5}{*}{$p = 10$} 
			& cdopt-lbfgs & 3.57e+01  & 106  & 116 & 116  & 8.42e-06     & 6.52e-14     & 0.20 \\
			& cdopt-cg & 3.57e+01  & 100  & 181 & 181  & 7.43e-06     & 1.11e-13     & 0.14 \\
			& cdopt-tr & 3.57e+01  & 20 & 21 & 19  & 8.23e-06     & 4.26e-15     & 0.13 \\\cline{2-9}
			& pymanopt-rcg   & 3.57e+01  & 62  & 62  & 62    & 6.74e-06     & 1.40e-15     & 0.15 \\
			& pymanopt-rtr  & 3.57e+01  & 13  & None  & None    & 2.84e-06     & 1.62e-15     & 0.55 \\\hline
			\multirow{5}{*}{$p = 20$} 
			& cdopt-lbfgs & 2.11e+02  & 166  & 175 & 175  & 8.10e-06     & 1.33e-14     & 1.60 \\
			& cdopt-cg & 2.11e+02  & 168  & 268 & 268  & 5.72e-06     & 5.62e-14     & 1.35 \\
			& cdopt-tr & 2.11e+02  & 31 & 32 & 28  & 2.58e-07     & 1.18e-15     & 2.71 \\\cline{2-9}
			& pymanopt-rcg   & 2.11e+02  & 160  & 160  & 160    & 9.40e-06     & 1.94e-15     & 1.16 \\
			& pymanopt-rtr  & 2.11e+02  & 11  & None  & None    & 8.31e-07     & 1.30e-15     & 1.50 \\\hline
			\multirow{5}{*}{$p = 30$} 
			& cdopt-lbfgs & 6.48e+02  & 301  & 324 & 324  & 7.61e-06     & 1.77e-15     & 7.82 \\
			& cdopt-cg & 6.48e+02  & 325  & 528 & 528  & 8.28e-06     & 5.76e-15     & 4.36 \\
			& cdopt-tr & 6.48e+02  & 32 & 33 & 30  & 4.01e-07     & 1.67e-15     & 3.04 \\\cline{2-9}
			& pymanopt-rcg   & 6.48e+02  & 458  & 458  & 458    & 2.39e-05     & 2.45e-15     & 5.88 \\
			& pymanopt-rtr  & 6.48e+02  & 9  & None  & None    & 1.60e-07     & 1.84e-15     & 2.55 \\\hline
			\multirow{5}{*}{$p = 50$}
			& cdopt-lbfgs & 2.81e+03  & 407  & 425 & 425  & 8.02e-06     & 3.10e-14     & 16.34 \\
			& cdopt-cg & 2.81e+03  & 431  & 682 & 682  & 4.75e-06     & 1.43e-15     & 7.41 \\
			& cdopt-tr & 2.81e+03  & 37 & 38 & 35  & 9.33e-07     & 1.91e-15     & 6.54 \\\cline{2-9}
			& pymanopt-rcg   & 2.81e+03  & 423  & 423  & 423    & 4.13e-05     & 2.27e-15     & 8.40 \\
			& pymanopt-rtr  & 2.81e+03  & 10  & None  & None    & 3.07e-10     & 3.87e-15     & 8.56 \\\hline
			\multirow{5}{*}{$p = 100$}
			& cdopt-lbfgs & 2.16e+04  & 732  & 773 & 773  & 2.70e-05     & 8.30e-15     & 38.17 \\
			& cdopt-cg & 2.16e+04  & 999  & 1632 & 1620  & 4.86e-05     & 2.75e-15     & 25.77 \\
			& cdopt-tr & 2.16e+04  & 58 & 59 & 56  & 6.50e-07     & 2.41e-15     & 22.04 \\\cline{2-9}
			& pymanopt-rcg   & 2.16e+04  & 854  & 854  & 854    & 5.28e-04     & 3.80e-15     & 25.24 \\
			& pymanopt-rtr  & 2.16e+04  & 12  & None  & None    & 1.55e-08     & 5.52e-15     & 23.49 \\\hline
		\end{tabular} 
	\end{table}

	\begin{table}[htbp]
		\centering
		\tiny
		\caption{Test results on nonlinear eigenvalue problems with $n = 3000$. }
		\label{Table_NEP_n3000}
		\begin{tabular}{l|l|ccccccc}
			\hline
			\multicolumn{2}{c|}{}                  & fval & iter & eval\_f & eval\_grad & stat & feas & time (s)  \\ \hline
			\multirow{5}{*}{$p = 10$} 
			& cdopt-lbfgs & 3.57e+01  & 128  & 140 & 140  & 7.54e-06     & 1.42e-14     & 2.04 \\
			& cdopt-cg & 3.57e+01  & 114  & 201 & 201  & 8.14e-06     & 1.60e-13     & 0.97 \\
			& cdopt-tr & 3.57e+01  & 24 & 25 & 23  & 2.84e-08     & 5.32e-16     & 1.58 \\\cline{2-9}
			& pymanopt-rcg   & 3.57e+01  & 81  & 81  & 81    & 9.25e-06     & 8.93e-16     & 0.89 \\
			& pymanopt-rtr  & 3.57e+01  & 8  & None  & None    & 8.39e-09     & 1.01e-15     & 5.69 \\\hline
			\multirow{5}{*}{$p = 20$} 
			& cdopt-lbfgs & 2.11e+02  & 148  & 159 & 159  & 6.52e-06     & 7.46e-13     & 6.39 \\
			& cdopt-cg & 2.11e+02  & 145  & 209 & 209  & 9.22e-06     & 5.58e-13     & 2.81 \\
			& cdopt-tr & 2.11e+02  & 19 & 20 & 18  & 3.26e-06     & 3.15e-15     & 2.04 \\\cline{2-9}
			& pymanopt-rcg   & 2.11e+02  & 136  & 136  & 136    & 9.69e-06     & 1.94e-15     & 2.19 \\
			& pymanopt-rtr  & 2.11e+02  & 11  & None  & None    & 7.14e-07     & 1.46e-15     & 12.63 \\\hline
			\multirow{5}{*}{$p = 30$}
			& cdopt-lbfgs & 6.48e+02  & 239  & 255 & 255  & 9.56e-06     & 1.45e-14     & 11.04 \\
			& cdopt-cg & 6.48e+02  & 284  & 456 & 456  & 8.09e-06     & 8.29e-15     & 4.81 \\
			& cdopt-tr & 6.48e+02  & 32 & 33 & 29  & 4.42e-08     & 1.53e-15     & 6.03 \\\cline{2-9}
			& pymanopt-rcg   & 6.48e+02  & 374  & 374  & 374    & 2.42e-05     & 1.82e-15     & 7.10 \\
			& pymanopt-rtr  & 6.48e+02  & 9  & None  & None    & 3.78e-08     & 2.46e-15     & 19.06 \\\hline
			\multirow{5}{*}{$p = 50$}
			& cdopt-lbfgs & 2.81e+03  & 396  & 418 & 418  & 9.48e-06     & 7.22e-15     & 35.26 \\
			& cdopt-cg & 2.81e+03  & 293  & 462 & 462  & 9.54e-06     & 5.77e-15     & 8.45 \\
			& cdopt-tr & 2.81e+03  & 46 & 47 & 41  & 5.19e-07     & 1.70e-15     & 13.47 \\\cline{2-9}
			& pymanopt-rcg   & 2.81e+03  & 328  & 328  & 328    & 2.90e-05     & 2.81e-15     & 9.19 \\
			& pymanopt-rtr  & 2.81e+03  & 9  & None  & None    & 5.94e-06     & 4.28e-15     & 23.60 \\\hline
			\multirow{5}{*}{$p = 100$}
			& cdopt-lbfgs & 2.16e+04  & 704  & 738 & 738  & 9.84e-05     & 2.17e-14     & 98.69 \\
			& cdopt-cg & 2.16e+04  & 645  & 1014 & 1014  & 5.86e-05     & 7.02e-14     & 38.70 \\
			& cdopt-tr & 2.16e+04  & 48 & 49 & 47  & 9.30e-08     & 2.16e-15     & 71.15 \\\cline{2-9}
			& pymanopt-rcg   & 2.16e+04  & 580  & 580  & 580    & 2.33e-04     & 3.82e-15     & 38.05 \\
			& pymanopt-rtr  & 2.16e+04  & 9  & None  & None    & 1.63e-08     & 5.73e-15     & 75.66 \\\hline
		\end{tabular} 
	\end{table}

	\subsection{Training orthogonally constrained neural network}
	In this subsection, we test the numerical performance of \CDOpt on training orthogonally constrained deep neural networks (DNNs). Moreover, we compare the numerical performance of \CDOpt with existing Riemannian optimization packages \Pkg{GeoTorch} and \Pkg{McTorch}. It is worth mentioning that although \Pkg{Geoopt} is designed for training neural networks, it is not easy to be used in practice. When applied to training neural networks, \Pkg{Geoopt} requires the users to write the layers themselves from \Pkg{PyTorch.nn} modules, including designing the \codeobj{\_\_init\_\_()} function that specifies the constrained manifold, the \codeobj{forward()} function that involves geometrical materials, and the \codeobj{reset\_parameters()} function that reset the parameters to generate a feasible point for training.  Therefore, we compare \CDOpt only with \Pkg{GeoTorch} and \Pkg{McTorch}. 
	
	Our first test example is to train a modified LeNet \cite{lecun1989backpropagation} for classification on the EMNIST-bymerge dataset \cite{emnist2017emnist}, which is a set of handwritten character digits derived from the NIST Special Database 19 and converted to a $28\times28$ pixel image format and dataset structure that directly matches the MNIST dataset. The EMNIST-bymerge dataset has $697932$ training samples and $116323$ test samples in $47$ classes.  As illustrated in Figure \ref{fig:lenet}, the network has $3$ convolution layers and $3$ fully-connected layers, and we impose the orthogonal constraints on the weight matrix of the first fully-connected layer. Moreover, we also choose to train VGG19 \cite{simonyan2014very} for classification on the CIFAR-10 dataset \cite{cifar2009learning}. CIFAR-10 dataset has $50000$ training samples and $10000$ test samples that are divided into $10$ classes. The VGG19 network contains $16$ convolution layers and $3$ fully-connected layers as illustrated in Figure \ref{fig:vgg}.   Furthermore, we train Densenet121 \cite{huang2017densely} for classification on the CIFAR-100 dataset \cite{cifar2009learning}, which has 100 classes in total, with $500$ training images and $100$ testing images in each class. Similar to LeNet, we impose the orthogonal constraints on the weight matrix of the first fully-connected layer of VGG19 and Densenet.

	\begin{figure}[htbp]
		\centering
		\subfigure[LeNet]{
			\begin{minipage}[t]{1\linewidth}
				\centering
				\includegraphics[width=\linewidth]{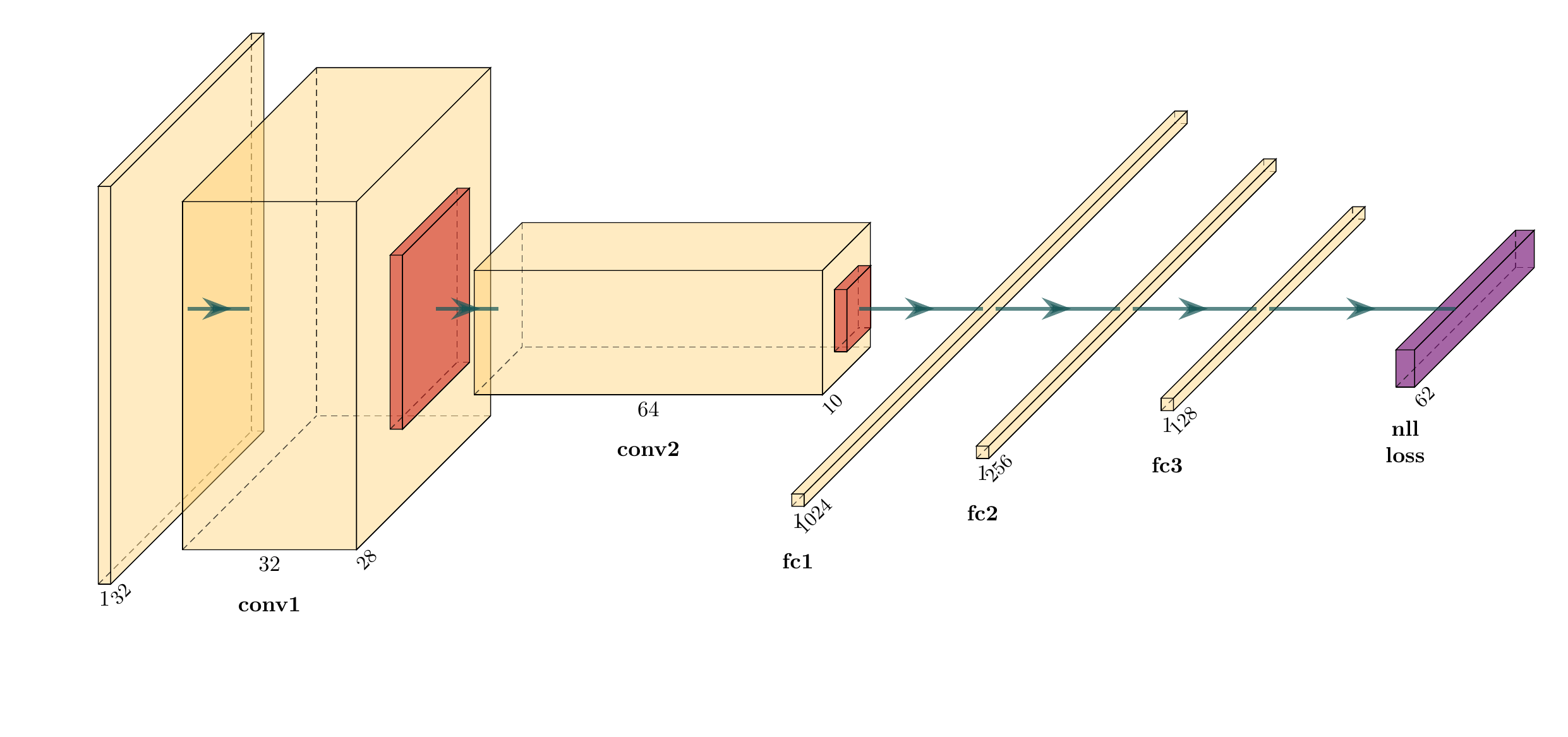}
				\label{fig:lenet}
			\end{minipage}%
		}%
		
		\subfigure[VGG19]{
			\begin{minipage}[t]{1\linewidth}
				\centering
				\includegraphics[width=\linewidth]{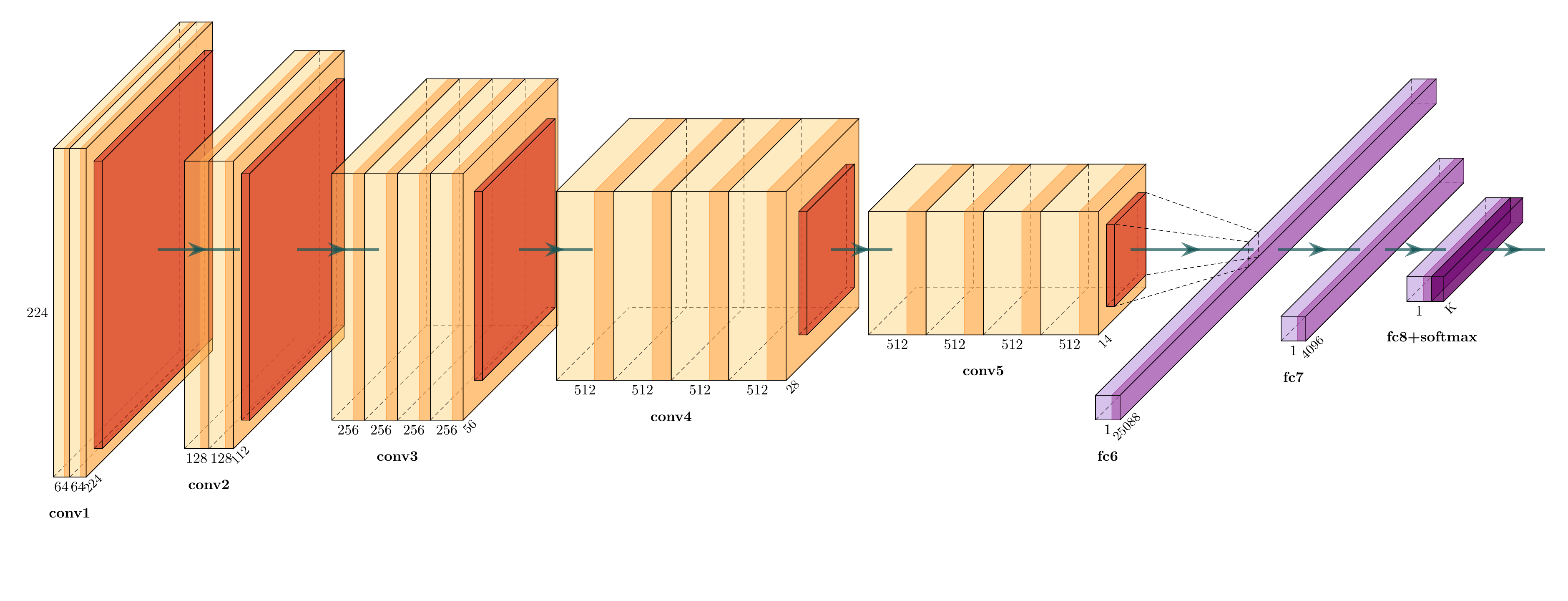}
				\label{fig:vgg}
			\end{minipage}%
		}%
		
		\subfigure[Densnet121]{
			\begin{minipage}[t]{1\linewidth}
				\centering
				\includegraphics[width=\linewidth]{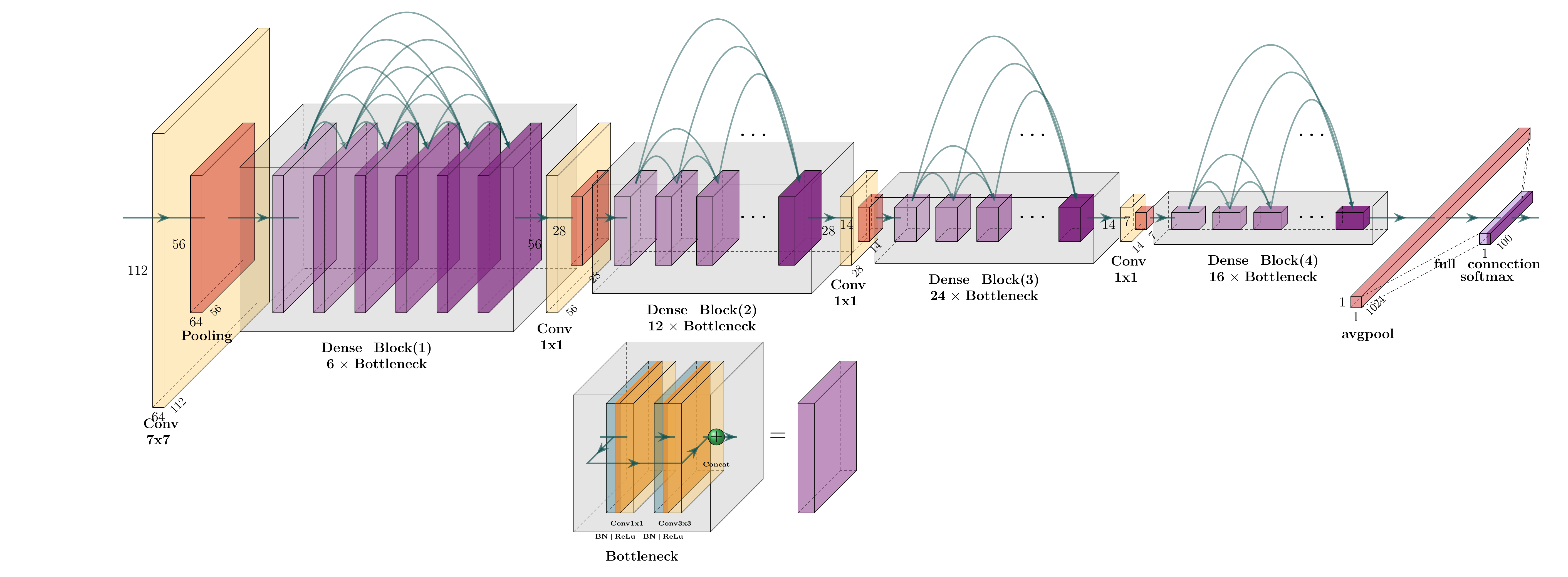}
				\label{fig:densenet}
			\end{minipage}%
		}%
		\caption{The structure of the tested neural networks.}
		\label{Fig:networks}
	\end{figure}
	
	In training these orthogonally constrained neural networks by \CDOpt, we choose the PyTorch build-in unconstrained optimizers SGD and Adam, together with the AdaMod, Lookahead, and Ranger from the  \Pkg{PyTorch-optimizer} packages. Developed from the trivialization approaches, \Pkg{GeoTorch} is also compatible to existing unconstrained optimizers, hence we try to apply SGD, Adam, AdaMod, Lookahead, and Ranger to train these orthogonally constrained networks through \Pkg{GeoTorch}. 
	However, among all the aforementioned solvers, only Riemannian SGD is provided in \Pkg{McTorch} package. 
	Therefore, we only use the Riemannian SGD from the \Pkg{McTorch} package to train these networks. 
	
	In training the modified LeNet, the batchsize is set as $64$, the total number of epochs is set as $30$, and the learning rate is set as $0.02$ for SGD with momentum parameter as $0.9$, $0.001$ for Adam \cite{kingma2014adam} and AdaMod \cite{ding2019adaptive}, $0.3$ for Ranger \cite{wright2019new} and $0.01$ for the base optimizer Yogi \cite{zaheer2018adaptive} in Lookahead \cite{zhang2019lookahead} optimizer. For training VGG19 and Densenet121, the batchsize is fixed as $256$, the total number of epochs is set as $300$, while the learning rate is set as $0.01$ for SGD (with momentum parameter as $0.9$), $3\times 10^{-4}$ for Adam, AdaMod, and DiffGrad, and $0.3$ for Ranger and $0.01$ for the base optimizer Yogi in Lookahead optimizer. Moreover, we choose the learning rate decrease strategy for all these optimizers and decrease the learning rate by multiplying a factor after every epoch. Such factor is set as $0.9$ in training LeNet, and $0.99$ in training VGG. Furthermore, when applying \CDOpt in training these networks, the penalty parameter $\beta$ in the constraint dissolving approaches is fixed as $0.01$ for all the test instances. 
	
	Table \ref{Table_DNN} exhibits the results of the numerical experiments, where the detailed meanings of the column headers are given in Table \ref{Table_header}. Compared with \Pkg{GeoTorch}, \CDOpt achieves similar accuracy and is compatible with various of unconstrained optimizers from \Pkg{PyTorch} and \Pkg{PyTorch-optimizer}. However, as computing the matrix exponential function and matrix inverse are expensive on GPU, \Pkg{GeoTorch} takes significantly  longer CPU time and costs more memory than \CDOpt.

	Furthermore, \CDOpt achieves similar accuracy as \Pkg{McTorch}. However, training manifold constrained neural networks by \CDOpt only requires matrix-matrix multiplication, while the optimizers in \Pkg{McTorch} involve computing the retractions in each iteration by singular value decomposition. Therefore, running unconstrained SGD by \CDOpt is slightly faster than running the Riemannian SGD by \Pkg{McTorch}.  Moreover, in the presence of difficulties in developing efficient Riemannian solvers, existing PyTorch-based Riemannian solvers are limited. For example, \Pkg{McTorch} only provides Riemannian SGD and Riemannian AdaGrad. In contrast, training manifold constrained neural networks by \CDOpt is highly adaptive to various existing unconstrained solvers.

	\begin{table}[htbp]
		\centering
		\tiny
		\caption{Numerical results on orthogonally constrained deep neural networks.}
		\label{Table_DNN}
		\scalebox{0.8}{
			\begin{tabular}{cc|ccc|ccc|ccc}
				\hline
				\multicolumn{2}{l|}{\multirow{2}{*}{}} & \multicolumn{3}{c|}{LeNet + EMNIST-byclass} & \multicolumn{3}{c|}{VGG 19 + CIFAR-10} & \multicolumn{3}{c}{Densenet121 + CIFAR-100} \\ 
				\multicolumn{2}{l|}{}                  & acc    & feas    & time/epoch    & acc     & feas     & time/epoch    & acc     & feas     & time/epoch    \\ \hline
				\multirow{3}{*}{SGD}     
				& \CDOpt    &91.07  &2.91e-03  &68.77  &92.81  &4.56e-03  &54.14  &76.74  &3.21e-03  &50.29  \\
				& \Pkg{GeoTorch} &91.04  &4.25e-03  &292.55  &\multicolumn{3}{c|}{out of memory}  &76.70&  4.87e-03 &56.20   \\
				& \Pkg{McTorch}  &91.05  &3.52e-03  &89.41  &92.70  &5.15e-03  &74.75  &76.45  &3.39e-03  &52.76  \\ \hline
				\multirow{3}{*}{Adam} 
				& \CDOpt    &91.02  &5.39e-03  &69.71  &92.70  &7.75e-03  &57.89  &72.13  &2.37e-03  &52.57  \\
				& \Pkg{GeoTorch} &91.05  &4.22e-03  &294.83  &\multicolumn{3}{c|}{out of memory}  &71.35  &4.61-03  & 54.96 \\
				& \Pkg{McTorch}  &\multicolumn{3}{c|}{not implemented} &\multicolumn{3}{c|}{not implemented}&\multicolumn{3}{c}{not implemented} \\ \hline
				\multirow{3}{*}{AdaMod} 
				& \CDOpt    &91.06  &3.81e-03  &82.53  &92.58  &4.38e-03  &58.78  &73.96  &2.95e-03  &56.10  \\
				& \Pkg{GeoTorch} &91.05  &4.32e-03  &319.62  &\multicolumn{3}{c|}{out of memory}  & 73.67   &4.83-03  & 58.14  \\
				& \Pkg{McTorch}  &\multicolumn{3}{c|}{not implemented} &\multicolumn{3}{c|}{not implemented}&\multicolumn{3}{c}{not implemented} \\ \hline
				\multirow{3}{*}{Ranger} 
				& \CDOpt    &91.11  &3.37e-03  &76.16  &92.76  &6.29e-03  &57.93  &73.28  &2.35e-03  &52.71  \\
				& \Pkg{GeoTorch} &90.96  &4.25e-03  &298.07  &\multicolumn{3}{c|}{out of memory}  &73.39   &4.49-03   &57.91    \\
				& \Pkg{McTorch}  &\multicolumn{3}{c|}{not implemented} &\multicolumn{3}{c|}{not implemented}&\multicolumn{3}{c}{not implemented} \\ \hline
				\multirow{3}{*}{Lookahead} 
				& \CDOpt    &91.07  &3.71e-03  &83.15  &92.97  &4.69e-03  &59.66  &74.38  &3.31e-03  &54.22  \\
				& \Pkg{GeoTorch} &91.01  &4.25e-03  &295.93  &\multicolumn{3}{c|}{out of memory}  &74.50   & 4.53e-03  &60.58   \\
				& \Pkg{McTorch}  &\multicolumn{3}{c|}{not implemented} &\multicolumn{3}{c|}{not implemented}&\multicolumn{3}{c}{not implemented} \\ \hline
			\end{tabular}
		}
	\end{table}

	\section{Conclusion}
	In this paper, we consider how to practically implement constraint dissolving approaches to solve Riemannian optimization problem \ref{Prob_Ori}. We first analyze the relationships between \ref{Prob_Ori} and \ref{Prob_Pen} under RCRCQ conditions, which is weaker than the assumptions in \cite{xiao2022constraint}.  Moreover, we propose a novel scheme for determining the threshold value of the penalty parameter $\beta$ in \ref{Prob_Pen}, which only involves the evaluations of $f$, $\A$, $c$, and their derivatives. We prove that with any penalty parameter $\beta$ greater than the suggested threshold value, \ref{Prob_Pen} and \ref{Prob_Ori} have the same first-order stationary points, second-order stationary points, and local minimums in a neighborhood of $\M$. In addition, based on the Monte Carlo sampling technique, we provide a practical scheme for choosing the penalty parameter $\beta$ in \ref{Prob_Pen}.
	
	Furthermore, we introduce the Python library package   \Pkg{CDOpt} for manifold optimization. \CDOpt serves as a complementary package to existing Riemannian optimization packages, in the sense that it enables direct implementations of various existing unconstrained optimization solvers for Riemannian optimization through the constraint dissolving approaches.  \Pkg{CDOpt} has user-friendly interfaces and supports various important features from its supported numerical backends, such as GPU/TPU supports, automatic differentiation, distributed training, JIT compilation, etc. Moreover,  \Pkg{CDOpt} only requires the expression of $c(x)$ in defining a new manifold. Therefore, users can enjoy great convenience in defining and solving a wide range of manifold constrained optimization problems, without any knowledge of the geometrical materials in differential geometry.  Furthermore, \CDOpt provides various predefined neural layers for \Pkg{PyTorch} and \Pkg{Flax}, which enables the users to build and train the manifold constrained neural networks only with minor modification to the standard \Pkg{PyTorch}/\Pkg{JAX} codes. Preliminary numerical experiments further demonstrate the high efficiency of \CDOpt, which can achieve comparable or superior performance 
	when compared with existing state-of-the-art Riemannian optimization packages.
	
	\Pkg{CDOpt} is still under active development, and some of the future works include adding supports for other numerical backends (e.g., \Pkg{TensorFlow}, \Pkg{MindSpore}, \Pkg{PaddlePaddle}), and integrate \Pkg{CDOpt} with more optimization frameworks, such as \Pkg{horovod} and \Pkg{Apex} for distributed training,  \Pkg{FATE} for federated learning, and \Pkg{BoTorch} for Bayesian optimization.

	\bibliography{ref}
	\bibliographystyle{spmpsci}      % mathematics and physical sciences

	\begin{acknowledgements}
		The numerical experiments in this paper are performed on the supercomputing system in the Supercomputing Center of Hangzhou City University. 
		
		The authors thank Professor Pierre-Antoine Absil for his helpful comments and discussions on Riemannian optimization approaches. The authors also thank the editors and anonymous reviewers for their valuable suggestions to improve the paper.
	\end{acknowledgements}

	\section*{Ethics Declarations}
	\subsection*{Funding}
	The research of Kim-Chuan Toh and Nachuan Xiao is supported by the Ministry of Education, Singapore, under its Academic Research Fund Tier 3 grant call (MOE-2019-T3-1-010). The research of Xiaoyin Hu is supported by the National Natural Science Foundation of China (Grant No. 12301408), Zhejiang Provincial Natural Science Foundation of China under Grant (No. LQ23A010002). The research of Xin Liu  is supported in part by the National Natural Science Foundation of China (No. 12125108, 11971466, 12288201, 12021001, 11991021) and Key Research Program of Frontier Sciences, Chinese Academy of Sciences (No. ZDBS-LY-7022).

	\subsection*{Conflict of interest}
	The authors declare that they have no conflict of interest.
	
	\subsection*{Availability of data and materials}
	All data analyzed during this study are publicly available. URLs are included in this published article.
	
	\subsection*{Code availability}
	The full code was made available for review. We remark that a set of packages were used in this study, that were either open source or available for academic use. Specific references are included in this published article.

\end{document}